\documentclass[a4paper,twoside,11pt]{article}
\usepackage[english]{babel}

\usepackage{amsmath,amsfonts,amssymb,amsthm}
\newtheorem{teo}{Theorem}
\newtheorem{lem}{Lemma}

 \title{{\bf    Riemann-Liouville  Operator in  Weighted $\mathbf{L_{p}}$ Spaces  via  the Jacobi Series  Expansion  }}

\author{Maksim \,V.~Kukushkin   \\ \\
   \small  \textit{Moscow State University of Civil Engineering}\\\textit{\small\textit{Russia, Moscow, 129337,}}\\
 \small  \textit{Kabardino-Balkarian Scientific Center, RAS, }\\\textit{\small\textit{Russia, Nalchik, 360051, kukushkinmv@rambler.ru}} }
\date{}
\begin{document}

\maketitle

\begin{abstract}
In this paper we  use the   orthogonal  system of the  Jacobi  polynomials as a tool  to study   the Riemann-Liouville    fractional integral  and derivative  operators  on a  compact of the real axis.
This approach has some advantages and allows us to complete  the previously known results of the fractional calculus theory by means of  reformulating them  in  a new quality.
     The proved  theorem on   the   fractional integral operator action   is  formulated in terms of   the Jacobi series    coefficients and is of particular interest.   We obtain a sufficient condition  for  a representation of a function  by the fractional integral in terms of   the  Jacobi series  coefficients.
      We   consider  several modifications  of the Jacobi  polynomials what gives  us an opportunity to study the  invariant property of the Riemann-Liouville  operator. In this direction,   we have shown that  the fractional integral  operator acting in the weighted  spaces of Lebesgue  square integrable  functions
   has  a sequence of the  included  invariant subspaces.

\end{abstract}
\begin{small}\textbf{Keywords:} Fractional derivative; fractional integral;  Riemann-Liouville  operator;  Jacobi  polynomials; Legendre  polynomials;
  invariant subspace.\\\\
{\textbf{MSC} 26A33; 47A15; 47A46; 12E10.}
\end{small}

\section{Introduction}
First, in this paper  we  aim  to reformulate  the   well-known theorems on  the  Riemann-Liouville  operator action   in terms of    the Jacobi  series  coefficients. In spite of that   this type of  problems was well studied by  such  mathematicians   as  Rubin B.S. \cite{firstab_lit: Rubin},\cite{firstab_lit: Rubin 1},\cite{firstab_lit: Rubin 2}, Vakulov B.G. \cite{firstab_lit: Vaculov},   Samko S.G.  \cite{firstab_lit: Samko M. Murdaev},\cite{firstab_lit: Samko Vakulov B. G.}, Karapetyants N.K.
\cite{firstab_lit: Karapetyants N. K. Rubin B. S. 1},\cite{firstab_lit: Karapetyants N. K. Rubin B. S. 2}
  (the results of \cite{firstab_lit: Karapetyants N. K. Rubin B. S. 1},\cite{firstab_lit: Rubin},\cite{firstab_lit: Rubin 1} are also  presented in \cite{firstab_lit:samko1987})   in several spaces and for various generalizations of the fractional integral  operator,    the method suggested in this work   allows us to notice   interesting properties  of the fractional integral and fractional derivative operators.  We suggest using   properties of the  Jacobi polynomials for studying the Riemann-Liouville operator, but we should make a remark that this idea was previously used  in the following papers \cite{firstab_lit Saad},\cite{firstab_lit:  Indian},\cite{firstab_lit:  Iranian society},\cite{firstab_lit:E.H. Doha},\cite{firstab_lit: Fr. oder Leg. functions},\cite{firstab_lit: SHENG CHEN}.
 For instance: in the papers \cite{firstab_lit:  Iranian society},\cite{firstab_lit:E.H. Doha} the   operational matrices of the Riemann-Liouville fractional integral and the Caputo fractional derivative  for shifted Jacobi
polynomials were  considered, in the paper \cite{firstab_lit:  Indian} the   fractional derivative formula was obtained applicably to the  general class of polynomials introduced by Srivastava, in the paper \cite{firstab_lit: Fr. oder Leg. functions}    a general formulation for the fractional-order Legendre functions   was
constructed to obtain the solution of the fractional order differential equations. Also, which  is interesting in  itself, the  fractional calculus theory was applied in \cite{firstab_lit(JFCA)Multivariabl},\cite{firstab_lit:  Abstract and Applied Analysis  ekim B},\cite{firstab_lit: Gogovcheva} to study the Jacobi polynomials.
However, our main  interest lies  in a rather different  field of studying   the mapping theorems for the Riemann-Liouville operator via the Jacobi polynomials.   This approach gives us such an advantage as getting results in terms of the Jacobi series coefficients, let alone the concrete achievements.
 The central point of our method of studying is to use the basis property of the Jacobi polynomials system. In  this way we aim to obtain a sufficient condition of existence and uniqueness of  the Abel equation solution  with the right part belonging to the weighted space of Lebesgue p-th integrable functions. Also, the usage of the weak topology gives us an opportunity to cover some cases in the mapping theorems  that were not previously obtained.
  Besides, having filled some conditions  gaps and formulated  the unified result,  we  aim to  systematize the mapping theorems established in  the monograph \cite{firstab_lit:samko1987}.

Secondly, we notice that
the question on existence of a non-trivial  invariant subspace for an arbitrary linear operator acting in a Hilbert space is still relevant for today. In 1935  J. von Neumann  proved that an arbitrary non-zero compact operator acting in a Hilbert space  has a  non-trivial invariant subspace \cite{firstab_lit:Aronszajne}. This approach  had got the further generalizations in the works \cite{firstab_lit:Bernstain},\cite{firstab_lit:Halmos}, but  the established  results are based on  the compact property of the operator. In the  general case the results \cite{firstab_lit:Lomonosov},\cite{firstab_lit:Nagy} are of particular   interest. The overview of   results in this direction can be found in \cite{firstab_lit:Helson},\cite{firstab_lit:Danford},\cite{firstab_lit:Halmos 1}. Due to many difficulties in solving  this  problem in the general case,  some scientists have paid attention to  special cases and one of these cases was the Volterra integral operator acting in the space  of  Lebesgue     square-integrable   functions on a compact of the  real axis.
 The  invariant subspaces of this operator were carefully studied and described in the papers \cite{firstab_lit:Brodsky},\cite{firstab_lit:Donoghue},\cite{firstab_lit:Kalisch}.
   We make an attempt  to study     invariant subspaces of the Riemann-Liouville fractional integral operator acting in the  weighted  space  of  Lebesgue square-integrable  functions on a compact of the  real axis. In this regard the following question is relevant: whether the Riemann-Liouville fractional integral has such an invariant subspace on which one would be selfadjoint.

   The paper is organized as follows: In the second  section the auxiliary formulas of fractional calculus are given as well as  a brief remark on the Jacobi polynomials system basis property. In the  third  section the main results are presented,   the mapping theorems established in  the monograph \cite{firstab_lit:samko1987}  were  systematized and  reformulated in terms of the Jacoby series coefficients,  the invariant subspaces of the Riemann-Liouville operator were studied. The   conclusions are given   in the  fourth  section.
\section{Preliminaries }
\subsection{Some fractional calculus  formulas }

Throughout this paper   we consider  complex functions of a real variable, we use the following denotation for   weighted complex  Lebesgue spaces     $L_{p}(I , \omega),\,1\leq p<\infty,$  where $I=(a,b)$ is an interval  of the real axis and the  weighted function $\omega$ is a real-valued function. Also we use the denotation $p'=p/(p-1).$ If $\omega=1,$ then we use the     notation  $L_{p}(I).$
Using the denotations of the paper \cite{firstab_lit:samko1987}, let us  define  the left-side,  right-side fractional integrals and   derivatives  of   real order respectively
 $$
\left(I_{a+}^{\alpha}f\right)(x)=\frac{1}{\Gamma(\alpha)} \int\limits_{a}^{x}\frac{f(t)}{(x-t)^{1-\alpha}}\,dt,\;\left(I_{b-}^{\alpha}f\right)(x)=\frac{1}{\Gamma(\alpha)} \int\limits_{x}^{b}\frac{f(t)}{(t-x)^{1-\alpha}}\,dt,\;f\in L_{1}(I);
 $$
$$
\left(D^{\alpha}_{a+}f\right)(x)=\frac{d^{n}}{dx^{n}} \left(I_{a+}^{n-\alpha}f\right)(x),\,f\in I_{a+}^{\alpha}(L_{1});   \;                  \left(D^{\alpha}_{b-}f\right)(x)= (-1)^{n}\frac{d^{n}}{dx^{n}}\left( I_{a+}^{n-\alpha}f\right)(x),\,
f\in I_{b-}^{\alpha}(L_{1}),
$$
$$
\,\alpha\geq0,\,n=[\alpha]+1,
$$
where $I_{a+}^{\alpha}(L_{1}),\,I_{b-}^{\alpha}(L_{1})$ are the classes of functions which can be represented by the fractional integrals    (see\cite[p.43]{firstab_lit:samko1987}).
    Further, we   use as a domain of definition of the   fractional differential   operators   mainly the set of   polynomials on which these operators are well defined.   We   use the shorthand notation $L_{2}:=L_{2}(I)$ and  denote by $(\cdot,\cdot) $ the inner product on  the Hilbert space $L_{2}(I).$
Using   Definition 1.5 \cite[p.4]{firstab_lit:samko1987}  we consider the space
$H_{0}^{\lambda}(\bar{I},r):=\{f:\,f(x)r(x)\in H^{\lambda}(\bar{I}),\, f(a)r(a)=f(b)r(b)=0\}$ endowed  with the norm
$$
\| f\|_{H_{0}^{\lambda}(\bar{I}\!,\,r)}=
\max\limits_{x\in I}|f(x)r(x)|+\sup\limits_{ \stackrel{x_{1},x_{2}\in I }{x_{1}\neq x_{2}}} \frac{|f(x_{1})r(x_{1})-f(x_{2})r(x_{2})|}{|x_{1}-x_{2}|^{\lambda}},\;r(x)=(x-a)^{\beta}(b-x)^{\gamma},\,\beta,\gamma\in \mathbb{R}.
$$
Denote   by   $ C,C_{i} ,\;i\in \mathbb{N} $   positive real  constants.   We mean that the values  of $C$   can be different in     various parts of    formulas,  but  the values of $C_{i} ,\;i\in \mathbb{N} $ are  certain.
 We use   the following  special denotation
$$
\binom{\eta} {\mu}:=\Gamma(\eta+1)/\Gamma(\eta-\mu+1),\;\eta,\mu\in \mathbb{R},\,\mu\neq-1,-2,...\,.
$$
Further, we need the following formulas for    multiple integrals. Note that  under the  assumption   $\varphi\in L_{1}(I),$ we have
\begin{equation}\label{1}
 \frac{1}{\Gamma( \alpha-m)}\underbrace{\int\limits_{a}^{x}dx\int\limits_{a}^{x}dx...\int\limits_{a}^{x}}_{\text{m+1 integrals} }\varphi(t)(x-t)^{ \alpha -m -1}dt=
\frac{1}{\Gamma( \alpha)}\int\limits_{a}^{x}(x-t)^{\alpha-1}\varphi(t)dt ;
$$
$$
 \frac{1}{\Gamma( \alpha-m)}\underbrace{\int\limits_{x}^{b}dx\int\limits_{x}^{b}dx...\int\limits_{x}^{b}}_{\text{m+1 integrals} }\varphi(t)(t-x )^{ \alpha -m -1}dt=
\frac{1}{\Gamma( \alpha)}\int\limits_{x}^{b}(t-x)^{\alpha-1}\varphi(t)dt,\;
m=\begin{cases}[\alpha],\, \alpha \in \mathbb{R}^+ \setminus \mathbb{N},\\
[\alpha]-1,\;\;\;\alpha \in  \mathbb{N}.
\end{cases}
\end{equation}
Suppose  $f(x)\in AC^{n}(\bar{I}), \,n\in \mathbb{N};$ then using the previous formulas  we have  the  representations
$$f(x)=\frac{1}{(n-1)!}\int\limits_{a}^{x}(x-t)^{n-1}f^{(n)}(t)dt+\sum\limits_{k=0}^{n-1} \frac{f^{(k)}(a)}{k!} (x-a)^{k};
$$
$$f(x)=\frac{(-1)^{n} }{(n-1)!}\int\limits_{x}^{b}(t-x)^{n-1}f^{(n)}(t)dt+\sum\limits_{k=0}^{n-1}(-1)^{k} \frac{f^{(k)}(b)}{k!} (b-x)^{k}.
$$
 Now assume that $ n=[  \alpha]+1$ in the previous formulas, then due to   Theorem 2.5  \cite[p.46]{firstab_lit:samko1987} and   formulas  of the   fractional integral of a power function   (2.44),(2.45)  \cite[p.40]{firstab_lit:samko1987},
 we have  in the left-side case
\begin{equation}\label{2}
(D_{a+}^{\alpha}f)(x)=
\sum\limits_{k=0}^{n-1}\frac{f^{(k)}(a)}{\Gamma(k+1-\alpha)}(x-a)^{k-\alpha}+\frac{1}{\Gamma(n-\alpha)}\int\limits_{a}^{x}\frac{f^{(n)}(t)}{(x-t)^{\alpha-n+1}}dt,
\end{equation}
 in the right-side case
\begin{equation}\label{3}
(D^{\alpha}_{b- }f)(x)=\sum\limits_{k=0}^{n-1}(-1)^{ k}\frac{f^{(k)}(b)}{ \Gamma( k+1-\alpha )}  (b-x)^{ k-\alpha}+\frac{(-1)^{n}}{\Gamma(n-\alpha)}\int\limits_{x}^{b}\frac{f^{(n)}(t)}{(t-x)^{\alpha-n+1}}dt.
\end{equation}

\subsection{Riemann-Liouville  operator via the Jacobi polynomials}

The orthonormal system of the  Jacobi  polynomials is denoted by
$$
p_{n}^{\,(\beta,\gamma)}(x)= \delta_{n} (\beta,\gamma) \, y^{\,(\beta,\gamma)}_{n}(x),\,n\in \mathbb{N}_{0},
$$
where  the  normalized multiplier $\delta_{n} (\beta,\gamma)$ is  defined by the formula
$$
\delta_{n}  (\beta,\gamma) =(-1)^{n}\frac{\sqrt{\beta+\gamma+2n+1}}{(b-a)^{n+(\beta+\gamma+1)/2}}\cdot \sqrt{\frac{ \Gamma(\beta+\gamma+n+1)}{n!\Gamma(\beta +n+1)\Gamma( \gamma+n+1)}} \;,\,
$$
$$
\delta_{0} (\beta,\gamma) =  \frac{1}{\sqrt{\Gamma(\beta  +1)\Gamma( \gamma +1)}}\,,\;\beta+\gamma+1=0,
$$
the orthogonal polynomials $y^{( \beta,\gamma)}_{n}$ are defined by the formula
$$
y^{( \beta,\gamma)}_{n}(x) =   (x-a)^{-\beta}(b-x)^{-\gamma}\frac{d^{n}}{dx^{n}}\left[(x-a)^{\beta+n}(b-x)^{\gamma+n}\right],\, \beta,\gamma>-1.
$$
For convenience, we use the following  functions
$$ \varphi^{(\beta,\gamma)} _{n}(x)=(x-a)^{ n+\beta} (b-x)^{ n+\gamma}.
$$
If   misunderstanding does not appear,     we  will  use  the shorthand denotations  in   various parts of this work
$$
p_{n}^{\,(\beta,\gamma)}(x):=p_{n} (x),\,y^{(\beta,\gamma)}_{n}(x):=y_{n}(x),\,\varphi^{(\beta,\gamma)} _{n}(x):=\varphi _{n}(x),\,\delta_{n}(\beta,\gamma):=\delta_{n}.
$$
In such  cases  we would like reader see carefully the denotations corresponding to a concrete paragraph.
  Specifically, in the case of the Jacobi  polynomials,   when $ \beta=\gamma=0,$ we have the  Legendre  polynomials.
If we consider the Hilbert space $L_{2}(I),$ then the  Legendre orthonormal system has a basis property due to
the general property of complete  orthonormal systems in Hilbert  spaces, but the question on the basis property of the Legendre  system for an  arbitrary $p\geq1,\,p\neq2$ had been  still relevant until    half of the last century. In the direction of   solving this problem the following works are  known  \cite{firstab_lit:H. Newman},\cite{firstab_lit:H. Pollard 1},\cite{firstab_lit:H. Pollard 2},\cite{firstab_lit:H. Pollard 3}.  In particular, in the paper \cite{firstab_lit:H. Pollard 1} Pollard H.   proved that the Legendre system has a basis property in the case $4/3<p<4 $ and for  the values of $p\in [1,4/3]\cup[4,\infty)$ the
Legendre  system does not have  a  basis property in $L_{p}(I)$ space. The cases $p=4/3,p=4$ were considered by Newman J. and  Rudin W. in the  paper \cite{firstab_lit:H. Newman} where it   is proved that in these cases the Legendre  system also does not have  a basis property in $L_{p}(I)$ space.
It is worth noting  that   the criterion of a basis property for the Jacobi polynomials    was  proved by Pollard H.  in the work \cite{firstab_lit:H. Pollard 3}. In that paper   Pollard H. formulated  the theorem proposing   that the Jacobi polynomials have a basic property in the space $L_{p}(I_{0},\omega),\,I_{0}:=(-1,1),\;\beta,\gamma\geq -1/2,\,M(\beta,\gamma)<p<m(\beta,\gamma)$  and do not have a basis property, when $ p<M(\beta,\gamma)$ or $ p>m(\beta,\gamma),$ where
$$
m(\beta,\gamma)=4\min\left\{\frac{\beta+1}{2\beta+1},\frac{\gamma+1}{2\gamma+1}\right\},\;
M(\beta,\gamma)=4\max\left\{\frac{\beta+1}{2\beta+3},\frac{\gamma+1}{2\gamma+3}\right\}.
$$
However, this result was subsequently improved by Muckenhoupt B.  in the paper \cite{firstab_lit:Muckenhoupt}.
Note that  the  linear  transform
\begin{equation*}
l:[-1,1] \rightarrow [a,b],\;y = \frac{b-a}{2}\, x+\frac{b+a}{2}
\end{equation*}
shows us that    all results  of the orthonormal polynomials theory   obtained for the segment $[-1,1]$ are true  for  the segment $[a,b]\subset \mathbb{R}.$  We use the  denotation $S_{k}f:=\sum_{n=0}^{k}f_{n}p^{(\beta,\gamma)}_{n},\,k\in \mathbb{N}_{0},$ where $f_{n}$ are the  Jacobi series  coefficients of  the function $f.$
Consider the orthonormal    Jacobi  polynomials
$$
p^{(\beta,\gamma)}_{n} (x)= \delta_{n }  y_{n} (x)=\delta_{n } (x-a)^{-\beta} (b-x)^{-\gamma} \varphi^{(n)}_{n}(x),\,\beta,\gamma>-1/2,\,n\in \mathbb{N}_{0}.
$$
Further, we need some   formulas. Using the Leibnitz formula, we get
 \begin{equation}\label{4}
y_{n} (x)=\sum\limits_{i=0}^{n}(-1)^{i}C^{i}_{n}  \tbinom{n+\beta} {n-i}   (x-a)^{  i}\tbinom{n+\gamma} {i}   (b-x)^{ n -i}=
 \sum\limits_{i=0}^{n}(-1)^{n+i}C^{i}_{n}  \tbinom{n+\beta}i    (x-a)^{n -i }\tbinom{n+\gamma} {n-i}  (b-x)^{  i} .
\end{equation}
Using again  the Leibnitz formula,  we obtain
\begin{equation}\label{5}
y_{n}^{(k)}(x) =
 \sum\limits_{i=0}^{n}(-1)^{i}C^{i}_{n}  \tbinom{n+\beta} {n-i}   \tbinom{n+\gamma} {i}\sum\limits_{j=c}^{i} (-1)^{k+j} C^{j}_{k}\tbinom{ i}{ j}  (x-a)^{i-j}   \tbinom{n-i}{k-j} (b-x)^{n+j -i-k},\,
 \end{equation}
 where
 $
 c= \max\left\{0\,,k+i-n\right\},\;k\leq n\,.
 $
In accordance with \eqref{5}, we have
\begin{equation}\label{6}
y_{n}^{(k)}(a)=(-1)^{k}(b-a)^{  n -k  }\sum\limits_{i=0}^{n}C^{i}_{n}  \tbinom{n+\beta} {n-i}   \tbinom{n+\gamma} {i}  C^{i}_{k}    \tbinom{n-i}{k-i}  i!,\,k\leq n;
\end{equation}
$$
\!\!p_{n}^{(k)}(a) \! =\!\!\frac{(-1)^{n+k}\sqrt{\beta+\gamma+2n+1}}{(b-a)^{k+(\beta+\gamma+1)/2}} \cdot \sqrt{\frac{ \Gamma(\beta+\gamma+n+1)}{n!\Gamma(\beta +n+1)\Gamma( \gamma+n+1)}} \sum\limits_{i=0}^{n}C^{i}_{n}  \tbinom{n+\beta} {n-i}   \tbinom{n+\gamma} {i}  C^{i}_{k}    \tbinom{n-i}{k-i}  i!,\,k\leq n.
$$
 In the same way, we get
\begin{equation}\label{7}
y_{n}^{(k)}(x)
 =\sum\limits_{i=0}^{n}(-1)^{n+i}C^{i}_{n}  \tbinom{n+\beta}{i} \tbinom{n+\gamma} {n-i}\sum\limits_{j=c}^{i}(-1)^{i} C^{j}_{k} \tbinom{n-i}{k-j}  (x-a)^{n+j -i-k } \tbinom{ i}{ j}  (b-x)^{  i-j},\,k\leq n.
\end{equation}
Hence
\begin{equation}\label{8}
y_{n}^{(k)}(b)=(-1)^{n}(b-a)^{n -k }\sum\limits_{i=0}^{n}C^{i}_{n}  \tbinom{n+\beta}{i} \tbinom{n+\gamma} {n-i}  C^{i}_{k} \tbinom{n-i}{k-i}   i!,\,k\leq n;
\end{equation}
$$
p_{n}^{(k)}(b)  =\frac{ \sqrt{\beta+\gamma+2n+1}}{n!(b-a)^{k+(\beta+\gamma+1)/2}}\cdot \sqrt{\frac{n!\Gamma(\beta+\gamma+n+1)}{\Gamma(\beta +n+1)\Gamma( \gamma+n+1)}} \sum\limits_{i=0}^{n}C^{i}_{n}  \tbinom{n+\beta} {n-i}   \tbinom{n+\gamma} {i}  C^{i}_{k}    \tbinom{n-i}{k-i}  i!,\,k\leq n.
$$
Let    $\mathfrak{C}_{n}^{(k)}(\beta,\gamma):=(-1)^{n+k}p_{n}^{(k)}(a)(b-a)^{k},$ then $p_{n}^{(k)}(b)(b-a)^{k}=\mathfrak{C}_{n}^{(k)}(\gamma,\beta).$
  Using    the Taylor  series expansion  for the  Jacobi   polynomials, we get
$$
p^{(\beta,\gamma)}_{n}(x)= \sum\limits_{k=0}^{n}(-1)^{n+k}(b-a)^{-k}  \frac{\mathfrak{C}_{n}^{(k)}(\beta,\gamma)}{ k!}  (x-a)^{k}=    \sum\limits_{k=0}^{n} (-1)^{ k}   (b-a)^{-k}\frac{\mathfrak{C}_{n}^{(k)}(\gamma,\beta)}{ k!}   (b-x)^{k} .
$$
Applying the formulas (2.44),(2.45) of  the   fractional integral and   derivative  of a power  function  \cite[p.40]{firstab_lit:samko1987},  we obtain
$$
(I_{a+}^{\alpha}p_{n})(x)= \sum\limits_{k=0}^{n} (-1)^{n+k}(b-a)^{-k}   \frac{  \mathfrak{C}_{n}^{(k)}(\beta,\gamma) }{\Gamma(k+1+\alpha)}  (x-a)^{k+\alpha},
$$
$$
\;(I_{b-}^{\alpha}p_{n})(x)=
  \sum\limits_{k=0}^{n} (-1)^{ k} (b-a)^{-k} \frac{  \mathfrak{C}_{n}^{(k)}(\gamma,\beta) }{\Gamma(k+1+\alpha)}  (b-x)^{k+\alpha},\,\alpha\in(-1,1),
$$
here we used the formal denotation $I^{-\alpha}_{a+}:=D_{a+}^{\alpha}.$
 Thus, using integration by parts, we get
$$
\int\limits_{a}^{b}p_{m}(x)(I^{\alpha}_{a+}p_{n})(x)\omega(x)dx=\delta_{m}\int\limits_{a}^{b}\varphi^{(m)}_{m}(x)(I^{\alpha}_{a+}p_{n})(x) dx=
$$
$$
=-\delta_{m}\int\limits_{a}^{b}\varphi^{(m-1)}_{m}(x)(I^{\alpha}_{a+}p_{n})^{(1)}(x) dx=
 (-1)^{m}\delta_{m}\int\limits_{a}^{b}\varphi_{m}(x)(I^{\alpha}_{a+}p_{n})^{(m)}(x) dx=
 $$
$$
= (-1)^{m}\delta_{m}\int\limits_{a}^{b} \sum\limits_{k=0}^{n}(-1)^{n+k}   (b-a)^{-k} \frac{  \mathfrak{C}_{n}^{(k)}(\beta,\gamma) }{\Gamma(k+\alpha-m+1)}  (x-a)^{k+\alpha+\beta } (b-x)^{m+\gamma} dx=
$$
$$
      =(-1)^{n }\hat{\delta}_{m}\sum\limits_{k=0}^{n} (-1)^{ k}   \frac{  \mathfrak{C}_{n}^{(k)}(\beta,\gamma)B(\alpha+\beta+k+1,\gamma+m+1) }{\Gamma(k+\alpha-m+1)},
$$
where
$$
\hat{\delta}_{m}=(b-a)^{\alpha+(\beta +\gamma+1)/2}\sqrt{\frac{(\beta+\gamma+2m+1) \Gamma(\beta+\gamma+m+1)}{m! \Gamma(\beta +m+1)\Gamma( \gamma+m+1)}}.
$$
 In the same way, we get
$$
(p_{m},I_{b-}^{\alpha}p_{n})_{L_{2}(I\!,\,\omega)}=(-1)^{ m}\hat{\delta}_{m}\sum\limits_{k=0}^{n}  (-1)^{ k}  \frac{  \mathfrak{C}_{n}^{(k)}(\gamma,\beta)B(\alpha+\gamma+k+1,\beta+m+1) }{\Gamma(k+\alpha-m+1)}.
$$
Using  the denotation
$$
A^{\alpha,\beta,\gamma}_{mn}:= \hat{\delta}_{m}\sum\limits_{k=0}^{n}  (-1)^{ k}  \frac{  \mathfrak{C}_{n}^{(k)}(\beta,\gamma)B(\alpha+\beta+k+1,\gamma+m+1) }{\Gamma(k+\alpha-m+1)}\,,
$$
 we have
\begin{equation}\label{9}
 (p_{m},I_{a+}^{\alpha}p_{n})_{L_{2}(I\!,\,\omega)}=(-1)^{n}A^{\alpha,\beta,\gamma}_{mn},\;\;
 (p_{m},I_{b-}^{\alpha}p_{n})_{L_{2}(I\!,\,\omega)}=(-1)^{m}A^{\alpha,\gamma,\beta}_{mn}.
 \end{equation}
 We claim the following formulas without any  proof
   because of the    absolute  analogy    with the proof corresponding to the fractional integral operators
\begin{equation}\label{10}
 (p_{m},D_{a+}^{\alpha}p_{n})_{L_{2}(I\!,\,\omega)}=(-1)^{n}A^{-\alpha,\beta,\gamma}_{mn},\;\;
 (p_{m},D_{b-}^{\alpha}p_{n})_{L_{2}(I\!,\,\omega)}=(-1)^{m}A^{-\alpha,\gamma,\beta}_{mn}.
 \end{equation}
  Further, we  use the following denotations
\begin{equation}\label{11}
   A_{+}^{\alpha,\beta,\gamma}:=
\begin{pmatrix}A^{\alpha,\beta,\gamma}_{00}&-A^{\alpha,\beta,\gamma}_{01}&...\\
 A^{\alpha,\beta,\gamma}_{10}&-A^{\alpha,\beta,\gamma}_{11}&...\\
\cdot\\
\cdot\\ \cdot&&...
\end{pmatrix},\;
\; A_{-}^{\alpha,\gamma,\beta}:=
  \begin{pmatrix}A^{\alpha,\gamma,\beta}_{00}&A^{\alpha,\gamma,\beta}_{01}&...\\
 -A^{\alpha,\gamma,\beta}_{10}&-A^{\alpha,\gamma,\beta}_{11}&...\\
\cdot\\
\cdot\\ \cdot&&...
\end{pmatrix},\,\alpha\in \mathbb{R}.
\end{equation}
 This  allows us to consider  the  integro-differential operators in the matrix form of notation.

Throughout this paper the  results are formulated and proved for the left-side case.   One may reformulate them for the right-side case with no difficulty.

\section{Main results}
\subsection{Mapping theorems}
The following lemma aims to establish more simplified  and at the same time applicable   form of the results proven  in   Theorem   3.10 \cite[p.78]{firstab_lit:samko1987}, Theorem 3.12 \cite[p.81]{firstab_lit:samko1987}  and is devoted to the description of the operator $I_{a+}^{\alpha}$ action   in the space $L_{p}(I,\omega).$   More precisely, these theorems describe the action $I^{\alpha}_{a+}:L_{p}(I,\omega)\rightarrow L_{q}(I,r)$ with rather inconveniently formulated conditions, from  the point of view of operator theory, regarding to the weighted functions and indexes $p,q.$  To justify this claim, we can easily see that there are some cases in the theorems conditions  for which  the    bounded  action     $I^{\alpha}_{a+}:L_{p}(I,\omega)\rightarrow L_{p}(I,\omega),\,\alpha\in(0,1),\,\omega(x)=(x-a)^{\beta}(b-x)^{\gamma}$ does not follow easily   from the theorems,   for instance in the case   $2<p<1/(1-\alpha),  \,\beta\in \mathbb{R}, \,0<\gamma\leq \alpha p-1,$ the mentioned above bounded action of $I_{a+}^{\alpha}$ cannot be obtained by using the theorems   and   estimating, as we shall see further the proof of this fact requires to involve the weak topology methods.

\begin{lem}\label{L1}
  Suppose $ \omega(x)=(x-a)^{\beta}(b-x)^{\gamma},\;   \beta, \gamma\in \left[- 1/2,1/2\right],\,M(\beta,\gamma)<p<m(\beta,\gamma);$ then
  \begin{equation}\label{14.1}
  \|I_{a+}^{\alpha}f\|_{L_{p}(I\!,\,\omega)}\leq C\|f\|_{L_{p}(I\!,\,\omega)},\,f\in L_{p}(I,\omega),\,\alpha\in(0,1).
\end{equation}

 \end{lem}
\begin{proof}
 By  direct calculation, we can verify that $\beta$ satisfies  the inequality $2t^{2}+t-1<0.$  We see that
 $$
2t^{2}+t-1\leq0; \,2t^{2}+3t\leq2t+1;\,t\leq\frac{2t+1}{2t+3};\,t+1\leq 4\frac{t+1}{2t+3}.
$$
Let us  substitute  $\beta$   for $t,$    we have
$$
\beta+1\leq 4\frac{\beta+1}{2\beta+3}\leq M(\beta,\gamma)<p.
$$
Hence $\beta<p-1.$ We have   absolutely analogous reasoning     for $\gamma$ i.e. $\gamma<p-1.$ Let us consider the various  relations between $p$ and $\alpha.$

\noindent   i)  $ p<1/\alpha.$ If     $\gamma>\alpha p-1,$ then    in accordance with Theorem 3.10 \cite[p.78]{firstab_lit:samko1987}, we get
\begin{equation}\label{14.2}
\|I_{a+}^{\alpha}f\|_{L_{q}(I,r)}\leq C\|f\|_{L_{p}(I\!,\,\omega)},\,q=p/(1-\alpha p),\,r(x)=(x-a)^{\frac{\beta q}{p}}(b-x)^{\frac{\gamma q}{p}}.
\end{equation}
Using the H\"{o}lder inequality, we obtain
$$
\left(\int\limits_{a}^{b}\left|I_{a+}^{\alpha}f\right|^{p}\omega \,dx\right)^{\frac{1}{p}}=
\left(\int\limits_{a}^{b}\left|\omega^{\frac{1}{p}} I_{a+}^{\alpha}f\right|^{p}dx\right)^{\frac{1}{p}}\leq C
\left(\int\limits_{a}^{b}\left| I_{a+}^{\alpha}f\right|^{q}\omega^{\frac{q}{p}} dx\right)^{\frac{1}{q}}
 =C
\left(\int\limits_{a}^{b}\left| I_{a+}^{\alpha}f\right|^{q}r dx\right)^{\frac{1}{q}}.
$$
Combining this inequality with  \eqref{14.2}, we obtain \eqref{14.1}. If
$\gamma\leq\alpha p-1,$ then   we have the following reasoning
$$
\left(\int\limits_{a}^{b}\left|I_{a+}^{\alpha}f\right|^{p}\omega\,dx\right)^{\frac{1}{p}}=
\left(\int\limits_{a}^{b}\left|(x-a)^{\frac{\beta}{p}}I_{a+}^{\alpha}f(x)\right|^{p}(b-x)^{\gamma}\,dx\right)^{\frac{1}{p}}= $$
$$
=\left(\int\limits_{a}^{b}\left|(x-a)^{\frac{\beta}{p}}I_{a+}^{\alpha}f(x)\right|^{p}(b-x)^{\frac{\gamma}{\xi}}
(b-x)^{\frac{\gamma}{\xi'}}\,dx\right)^{\frac{1}{p}}=I_{1},\,\xi=1/(1-\alpha p).
$$
Using the H\"{o}lder inequality, we get
$$
I_{1}=\left(\int\limits_{a}^{b}\left|(x-a)^{\frac{\beta}{p}}(b-x)^{\frac{\gamma }{p\,\xi}}I_{a+}^{\alpha}f(x)\right|^{p}
(b-x)^{\frac{\gamma}{\xi'}}\,dx\right)^{\frac{1}{p}}\leq
$$
$$
\leq \left(\int\limits_{a}^{b}\left|(x-a)^{\frac{\beta}{p}}(b-x)^{\frac{\gamma }{p\,\xi}}I_{a+}^{\alpha}f(x)\right|^{p\xi}
 \,dx\right)^{\frac{1}{p\xi}}\times
 \left(\int\limits_{a}^{b}  (b-x)^{ \gamma  }
 \,dx\right)^{\frac{1}{p\xi'}}=
$$
$$
C\left(\int\limits_{a}^{b}\left|I_{a+}^{\alpha}f(x)\right|^{q}(x-a)^{\frac{\beta q}{p}}(b-x)^{ \gamma  }
 \,dx\right)^{\frac{1}{q}} \leq
 C\left(\int\limits_{a}^{b}\left|I_{a+}^{\alpha}f(x)\right|^{q}(x-a)^{\frac{\beta q}{p}}(b-x)^{ \nu  }
 \,dx\right)^{\frac{1}{q}},\,-1<\nu<\gamma.
$$
Applying Theorem 3.10 \cite[p.78]{firstab_lit:samko1987}, we obtain
$$
\left(\int\limits_{a}^{b}\left|I_{a+}^{\alpha}f(x)\right|^{q}(x-a)^{\frac{\beta q}{p}}(b-x)^{ \nu  }
 \,dx\right)^{\frac{1}{q}}\leq C \|f\|_{L_{p}(I\!,\,\omega)}.
$$
Hence \eqref{14.1} is fulfilled.

\noindent   ii) $1/\alpha<p.$ We have several cases.

\noindent a)
  $\gamma\leq0$ or $\gamma>\alpha p-1.$  If  $\gamma\leq0,$ then  applying Theorem 3.8 \cite[p.74]{firstab_lit:samko1987}, we obtain
$$
\|I^{\alpha}_{a+}f\|_{H_{0}^{\alpha-1/p}(I,r_{1})}\leq C \left(\int\limits_{a}^{b}\left| f(x)\right|^{p } (x-a)^{ \beta } dx\right)^{\frac{1}{p }}\leq C \left(\int\limits_{a}^{b}\left| f(x)\right|^{p }\omega(x) dx\right)^{\frac{1}{p }},
$$
where $r_{1}(x)=(x-a)^{\frac{\beta}{p }}.$
  We have the following estimate
$$
\left(\int\limits_{a}^{b}\left|I^{\alpha}_{a+}f\right|^{p }\omega (x)dx\right)^{\frac{1}{p }}=
\left(\int\limits_{a}^{b}\left|(x-a)^{\frac{\beta}{p }}I^{\alpha}_{a+}f\right|^{p } (b-x)^{ \gamma  }dx\right)^{\frac{1}{p'}}\leq
$$
$$
\leq \|I^{\alpha}_{a+}f\|_{H_{0}^{\alpha-1/p}(I,r_{1})} \left(\int\limits_{a}^{b}  (b-x)^{ \gamma  }dx\right)^{\frac{1}{p }}=C \|I^{\alpha}_{a+}f\|_{H_{0}^{\alpha-1/p}(I,r_{1})}.
$$
 Hence \eqref{14.1} is fulfilled.
If  $\gamma>\alpha p-1,$   then  we get
$$
\left(\int\limits_{a}^{b}\left|I^{\alpha}_{a+}f(x)\right|^{p}\omega(x)dx\right)^{\frac{1}{p}}=
\left(\int\limits_{a}^{b}\left|(x-a)^{\frac{\beta }{p}}(b-x)^{\frac{\gamma }{p }}
I^{\alpha}_{a+}f(x)\right|^{p}   dx\right)^{\frac{1}{p }}
 \leq C\|I^{\alpha}_{a+}f(x)\|_{H_{0}^{\alpha-1/p}(I,r_{1})},
$$
where $r_{1}(x)=(x-a)^{\frac{\beta }{p}}(b-x)^{\frac{\gamma }{p }}.$
  Applying Theorem 3.12 \cite[p.81]{firstab_lit:samko1987}, we obtain
$$
\|I^{\alpha}_{a+}f\|_{H_{0}^{\alpha-1/p }(I,r_{1})}\leq C \left(\int\limits_{a}^{b}\left| f(x)\right|^{p } \omega (x)dx\right)^{\frac{1}{p}}.
$$
Hence \eqref{14.1} is fulfilled.

 \noindent b)   $p\leq2,\,0<\gamma\leq \alpha p-1.$   As a consequence of the condition $p\leq2,$ we get
$\gamma-(\alpha p-1)>-1.$
  We obtain the estimate
$$
\left(\int\limits_{a}^{b}\left|I^{\alpha}_{a+}f\right|^{p}\omega(x)dx\right)^{\frac{1}{p}}=
\left(\int\limits_{a}^{b}
\left|(x-a)^{\frac{\beta}{p}}(b-x)^{ \frac{\theta}{p} }I^{\alpha}_{a+}f(x)\right|^{p}
(b-x)^{\gamma-\theta}   dx\right)^{\frac{1}{p}}\leq
$$
$$
\leq  \|I^{\alpha}_{b-}\varphi_{m} \|_{H_{0}^{\alpha-1/p}(I,r_{1})}\left(\int\limits_{a}^{b}
 (b-x)^{\gamma-\theta}  dx\right)^{\frac{1}{p}}  ,
$$
where $\theta=(\alpha p-1)+p\, \delta,\,\delta>0,\,      r_{1}(x)=(x-a)^{\frac{\beta}{p}}(b-x)^{ \frac{\theta}{p} }.$ Taking into account that   $\gamma-(\alpha p-1)>-1,$    we get for   sufficiently  small   $\delta>0$
$$
\left(\int\limits_{a}^{b}
 (b-x)^{\gamma-\theta}  dx\right)^{\frac{1}{p}} <\infty.
$$
Applying Theorem 3.12 \cite[p.81]{firstab_lit:samko1987}, we obtain
$$
\|I^{\alpha}_{a+}f\|_{H_{0}^{\alpha-1/p}(I,r_{1})}\leq C \left(\int\limits_{a}^{b}\left| f(x)\right|^{p}\omega(x) dx\right)^{\frac{1}{p}}.
$$
Hence \eqref{14.1} is fulfilled.\\
\noindent c) $p>2, \,0<\gamma\leq \alpha p-1.$ In this case    we should    consider various   subcases.

\noindent  1) $ p'>1/\alpha.$
If $\beta\geq0,$ then     we   note  that $\varphi^{(m)}_{m}(x)(b-x)^{-\gamma}\in L_{\infty}(I).$    Hence
\begin{equation}\label{14.3}
 \int\limits_{a}^{b}|\varphi^{(m)}_{m}(x)|^{p'} (b-x)^{\gamma(1-p')} dx <\infty.
\end{equation}
It is easily shown that
$$
\left(\int\limits_{a}^{b}\left|I^{\alpha}_{b-}\varphi^{(m)}_{m}\right|^{p'}\omega^{1-p'}(x)dx\right)^{\frac{1}{p'}}=
\left(\int\limits_{a}^{b}\left|(b-x)^{\frac{\gamma(1-p')}{p'}}I^{\alpha}_{b-}\varphi^{(m)}_{m}(x)\right|^{p'} (x-a)^{ \beta(1-p')  }dx\right)^{\frac{1}{p'}}\leq
$$
$$
\leq \|I^{\alpha}_{b-}\varphi^{(m)}_{m}(x)\|_{H_{0}^{\alpha-1/p'}(I,r_{1})} \left(\int\limits_{a}^{b}  (x-a)^{ \beta(1-p')  }dx\right)^{\frac{1}{p'}},
$$
where $r_{1}(x)=(b-x)^{ \gamma(1-p')/ p' }.$ Solving the quadratic  equality we can verify that
under the assumptions  $0<\beta\leq 1/2,$  we have
 $
 4 (\beta+1)/(2\beta+1)\leq (\beta+1)/\beta.
 $
 Since it can easily be checked that $p'<m(\beta,\gamma)\leq 4 (\beta+1)/(2\beta+1),$
then    $p'<(\beta+1)/\beta$ or $\beta(1-p')>-1.$ Hence
$$
 \left(\int\limits_{a}^{b}\left|I^{\alpha}_{b-}\varphi^{(m)}_{m}\right|^{p'}\omega^{1-p'}(x)dx\right)^{\frac{1}{p'}}\leq C\|I^{\alpha}_{b-}\varphi^{(m)}_{m}(x)\|_{H_{0}^{\alpha-1/p'}(I,r_{1})}.
$$
 It is obvious  that $\gamma(1-p')<p'-1.$  Combining  relation \eqref{14.3} and  Theorem 3.8 \cite[p.74]{firstab_lit:samko1987}, we obtain
$$
\|I^{\alpha}_{b-}\varphi^{(m)}_{m}(x)\|_{H_{0}^{\alpha-1/p'}(I,r_{1})}\leq C \left(\int\limits_{a}^{b}\left| \varphi^{(m)}_{m}(x)\right|^{p'} (b-x)^{ \gamma(1-p') }dx\right)^{\frac{1}{p'}}<\infty.
$$
  Since $\beta(1-p')\leq0,$ then
$$
\left(\int\limits_{a}^{b}\left| \varphi^{(m)}_{m}(x)\right|^{p'} (b-x)^{ \gamma(1-p') }dx\right)^{\frac{1}{p'}}\leq
(b-a)^{ \beta( p'-1)  }\left(\int\limits_{a}^{b}\left| \varphi^{(m)}_{m}(x)\right|^{p'}(x-a)^{ \beta(1-p')  } (b-x)^{ \gamma(1-p') }dx\right)^{\frac{1}{p'}}=
$$
$$
=(b-a)^{ \beta( p'-1)  }\left(\int\limits_{a}^{b}\left| p_{m}(x)\right|^{p'}(x-a)^{ \beta   } (b-x)^{ \gamma  }dx\right)^{\frac{1}{p'}}.
$$
Taking into account  the  above considerations, we obtain
\begin{equation}\label{14.4}
\left(\int\limits_{a}^{b}\left|I^{\alpha}_{b-}\varphi^{(m)}_{m}\right|^{p'}\omega^{1-p'}(x)dx\right)^{\frac{1}{p'}}\leq C \|p_{m}\|_{L_{p'}(I,\omega)},\,m\in \mathbb{N}_{0}.
\end{equation}
Thus, we get  $\omega^{-1}I^{\alpha}_{b-}\varphi^{(m)}_{m}\in L_{p'}(I,\omega).$
Using the H\"{o}lder inequality and  the previous reasoning, we get
$$
I_{2}= \left|\int\limits_{a}^{b}f(x)
dx\int\limits_{x}^{b}  \varphi^{(m)}_{m}(t)(t-x)^{\alpha-1} dt\right|=\left|\int\limits_{a}^{b}
\left\{\omega^{-1}(x)\int\limits_{x}^{b} \varphi^{(m)}_{m}(t)(t-x)^{\alpha-1} dt\right\}  f(x)  \,\omega(x)dx\right|\leq
$$
$$
\leq \left\{\int\limits_{a}^{b}|f(x)|^{p}\omega(x)dx\right\}^{1/p}\!\!
 \, \left\{\int\limits_{a}^{b}\left|\omega^{-1}(x)\int\limits_{x}^{b} \varphi^{(m)}_{m}(t)(t-x)^{\alpha-1} dt \right|^{p'} \omega(x) dx \right\}^{1/p'}\!\!\!\leq
$$
$$
\leq C\|f\|_{L_{p}(I\!,\,\omega)}\,\|p_{m}\|_{L_{p'}(I\!,\,\omega)}<\infty,\, f\in L_{p}(I,\omega),\;m\in \mathbb{N}_{0}.
$$
Hence in accordance with   the consequence  of the Fubini theorem, we get
 \begin{equation}\label{14.5}
  \left(  I^{\alpha}_{a+}f ,p_{m}    \right)_{L_{2}(I\!,\,\omega)}=\left( f ,\omega^{-1}I^{\alpha}_{b-}\varphi^{(m)}_{m} \right)_{L_{2}(I\!,\,\omega)},\;m\in \mathbb{N}_{0} .
 \end{equation}
Consider the functional
$$
l_{f}(p_{m})=\left(  I^{\alpha}_{a+}f ,p_{m}    \right)_{L_{2}(I\!,\,\omega)} =\left( f ,\omega^{-1}I^{\alpha}_{b-}\varphi^{(m)}_{m} \right)_{L_{2}(I\!,\,\omega)}.
$$
Applying  \eqref{14.4}, we obtain
 \begin{equation}\label{14.6}
|l_{f}(p_{m})|\leq C \|f\|_{L_{p}(I\!,\,\omega)} \|p_{m}\|_{L_{p'}(I\!,\,\omega)},\,m\in \mathbb{N}_{0}.
\end{equation}
We see that    the  previous inequality is true for all linear combinations
 \begin{equation}\label{14.7}
|l_{f}(\mathcal{L}_{m})|\leq C\|f\|_{L_{p}(I\!,\,\omega)} \|\mathcal{L}_{m}\|_{L_{p'}(I\!,\,\omega)},\, \mathcal{L}_{m}:=\sum\limits_{n=0}^{m}c_{n}p_{n},\,c_{n}={\rm const},\,m\in \mathbb{N}_{0}.
\end{equation}
Since it can easily be checked that $M(\beta,\gamma)< p'< m(\beta,\gamma),$  then in accordance with the results of the paper  \cite{firstab_lit:H. Pollard 3} the system $\{p_{m}\}_{0}^{\infty}$ has a basis property in the space $L_{p'}(I\!,\,\omega).$ Using this fact,   we   pass  to the limit in both sides of   inequality   \eqref{14.7}, thus  we get
\begin{equation}\label{14.8}
|l_{f}(g)|\leq C\|f\|_{L_{p}(I\!,\,\omega)} \|g\|_{L_{p'}(I\!,\,\omega)},\, \forall g\in L_{p'(I\!,\,\omega)}.
\end{equation}
 In the terms of the given above denotation we can write
$$
|\left(  I^{\alpha}_{a+}f ,g    \right)_{L_{2}(I\!,\,\omega)}|\leq C\|f\|_{L_{p}(I\!,\,\omega)} \|g\|_{L_{p'}(I\!,\,\omega)},\, \forall g\in L_{p'(I\!,\,\omega)}.
$$
In its turn, this inequality can be   rewritten  in the following form
$$
\left|\left(  \frac{\!\!I^{\alpha}_{a+}f}{\,\|f\|_{L_{p}(I\!,\,\omega)}}\, ,g \right)_{ \!\!L_{2}(I\!,\,\omega) }\right|\leq C \|g\|_{L_{p'}(I\!,\,\omega)},\, \forall g\in L_{p'(I\!,\,\omega)}.
$$
Hence  the set
$$
\mathcal{F}:=\left\{\frac{\!\!I^{\alpha}_{a+}f}{\,\|f\|_{L_{p}(I\!,\,\omega)}}\,,\,f\in L_{p}(I,\omega) \right\}
$$
is   weekly bounded. Therefore,  in accordance with the  well-known theorem this set is bounded with respect to the norm  $L_{p}(I,\omega).$ It implies that \eqref{14.1} holds.
If $\beta<0,$ then it is easy to show  that  $\beta(1-p')- \alpha p'+1>-1.$
Under the assumptions  $\beta(p'-1)\leq\alpha p'-1,$     we have
$$
\left(\int\limits_{a}^{b}\left|I^{\alpha}_{b-}\varphi_{m}\right|^{p'}\omega^{1-p'}(x)dx\right)^{\frac{1}{p'}}=
\left(\int\limits_{a}^{b}
\left|(x-a)^{ \frac{\theta}{p'} }(b-x)^{\frac{\gamma(1-p')}{p'}}I^{\alpha}_{b-}\varphi_{m}(x)\right|^{p'}
(x-a)^{\beta(1-p')-\theta}   dx\right)^{\frac{1}{p'}}\leq
$$
$$
\leq  \|I^{\alpha}_{b-}\varphi_{m}(x)\|_{H_{0}^{\alpha-1/p'}(I,r_{1})}\left(\int\limits_{a}^{b}
 (x-a)^{\beta(1-p')-\theta}   dx\right)^{\frac{1}{p'}}  ,
$$
where $\theta=(\alpha p'-1)+p'\delta,\,\delta>0,\,      r_{1}(x)=(x-a)^{ \theta/  p'}  (b-x)^{ \gamma(1-p')/ p' }.$ Hence  using  the condition $\beta(1-p')-\theta>-1$,  we get
for sufficiently  small $\delta>0$
$$
\left(\int\limits_{a}^{b}\left|I^{\alpha}_{b-}\varphi_{m}\right|^{p'}\omega^{1-p'}(x)dx\right)^{\frac{1}{p'}}\leq C
\|I^{\alpha}_{b-}\varphi_{m}(x)\|_{H_{0}^{\alpha-1/p'}(I,r_{1})}.
$$
On the other hand, under the  assumptions  $\beta(1-p')>\alpha p'-1,$ we can evaluate directly
$$
\left(\int\limits_{a}^{b}\left|I^{\alpha}_{b-}\varphi_{m}\right|^{p'}\omega^{1-p'}(x)dx\right)^{\frac{1}{p'}}=
\left(\int\limits_{a}^{b}
\left|(x-a)^{\frac{\beta(1-p')}{p'}}(b-x)^{\frac{\gamma(1-p')}{p'}}I^{\alpha}_{b-}\varphi_{m}(x)\right|^{p'}
   dx\right)^{\frac{1}{p'}}\leq
$$
$$
\leq  C\|I^{\alpha}_{b-}\varphi_{m}(x)\|_{H_{0}^{\alpha-1/p'}(I,r_{1})},
$$
where $r_{1}(x)=(x-a)^{\frac{\beta(1-p')}{p'}}(b-x)^{\frac{\gamma(1-p')}{p'}}.$
Applying Theorem 3.12 \cite[p.81]{firstab_lit:samko1987}, we obtain
$$
\|I^{\alpha}_{b-}\varphi_{m}(x)\|_{H_{0}^{\alpha-1/p'}(I,r_{1})}\leq C \left(\int\limits_{a}^{b}\left| \varphi_{m}(x)\right|^{p'} \omega^{1-p'}(x)dx\right)^{\frac{1}{p'}}=\left(\int\limits_{a}^{b}\left| p_{m}(x)\right|^{p'}\omega(x) dx\right)^{\frac{1}{p'}}.
$$
Hence  inequality  \eqref{14.4} holds. Arguing as above, we obtain   \eqref{14.1}.\\
\noindent 2) $ p'<1/\alpha.$
   We should apply  the  reasoning  used in  (i),   in this way
 we obtain easily   \eqref{14.4}. Further,  we get \eqref{14.1}  in  the way considered above.

\noindent   iii)  $\alpha=1/p.$   We already know that    due to the condition $M(\beta,\gamma)<p<m(\beta,\gamma),$ we have $\beta,\gamma<p-1.$ Let $p_{1}=p-\varepsilon,\,\varepsilon>0,\;\beta,\gamma<p_{1}-1.$   If   $\gamma\geq0,$ then we should use      the following reasoning
$$
\left(\int\limits_{a}^{b}  |I_{a+}^{1/p}f  |^{p}\omega dx\right)^{\frac{1}{p} }=\left(\int\limits_{a}^{b}
  |\omega^{\frac{1}{p_{1}}} I_{a+}^{1/p}f  |^{p}\omega^{1-\frac{p}{p_{1}}} dx\right)^{\frac{1}{p}}\leq
 C\left(\int\limits_{a}^{b}  |I_{a+}^{1/p}f  |^{q}\omega^{\frac{q}{p_{1}}} dx\right)^{\frac{1}{q}}\left(\int\limits_{a}^{b} \omega^{\left(1-\frac{p}{p_{1}}\right)\xi'} dx\right)^{\frac{1}{p\xi'}},
$$
where $q=p\,\xi,\, \xi=p_{1}/p(1-p_{1}p^{-1}).$ Thus for sufficiently  small $\varepsilon,$ we obtain
$$
\int\limits_{a}^{b} \omega^{\left(1-\frac{p}{p_{1}}\right)\xi'} dx<\infty.
$$
  Taking into account that    $\gamma>p_{1}p^{-1}-1$ and applying   Theorem 3.10 \cite[p.78]{firstab_lit:samko1987}, we get
$$
\left(\int\limits_{a}^{b}  |I_{a+}^{1/p}f (x) |^{p}\omega(x) dx\right)^{\frac{1}{p} }\leq C \left(\int\limits_{a}^{b}  |I_{a+}^{1/p}f(x)  |^{q}\omega^{\frac{q}{p_{1}}}(x) dx\right)^{\frac{1}{q}}\leq C \left(\int\limits_{a}^{b}  | f(x)  |^{p_{1}}\omega(x)  dx\right)^{\frac{1}{p_{1}}}.
$$
Thus noticing   that  $\|f\|_{L_{p_{1}}(I,\omega)}\leq C \|f\|_{L_{p}(I\!,\,\omega)},$ we obtain  \eqref{14.1}.
If $\gamma<0,$ then   we can choose $\varepsilon$ so that  $\gamma< p_{1}p^{-1}-1.$  We have the following reasoning
$$
\left(\int\limits_{a}^{b}\left|I_{a+}^{1/p}f\right|^{p}\omega\,dx\right)^{\frac{1}{p}}=
  \left(\int\limits_{a}^{b}\left|(x-a)^{\frac{\beta}{p_{1}}}I_{a+}^{1/p}f(x)\right|^{p}(x-a)^{\beta\left(1-\frac{ p}{p_{1}}\right)}  (b-x)^{\gamma\left(\frac{1}{\xi}+\frac{1}{\xi'}\right)}
  \,dx\right)^{\frac{1}{p}}=I_{1},
$$
where $q=p\,\xi,\, \xi=p_{1}/p(1-p_{1}p^{-1}).$   Using the H\"{o}lder inequality, we get
$$
I_{1}=\left(\int\limits_{a}^{b}\left|(x-a)^{\frac{\beta}{p_{1}}}(b-x)^{\frac{\gamma }{p\,\xi}}I_{a+}^{1/p}f(x)\right|^{p}
(x-a)^{\beta\left(1-\frac{ p}{p_{1}}\right)}(b-x)^{\frac{\gamma}{\xi'}}\,dx\right)^{\frac{1}{p}}\leq
$$
$$
\leq \left(\int\limits_{a}^{b}\left|(x-a)^{\frac{\beta}{p_{1}}}(b-x)^{\frac{\gamma }{p\,\xi}}I_{a+}^{1/p}f(x)\right|^{p\xi}
 \,dx\right)^{\frac{1}{p\xi}}\times
 \left(\int\limits_{a}^{b} (x-a)^{\beta\left(1-\frac{ p}{p_{1}}\right)\xi'} (b-x)^{ \gamma  }
 \,dx\right)^{\frac{1}{p\xi'}}.
$$
We can choose $\varepsilon$ so that we   have $\beta\left(1-  p/ p_{1} \right)\xi'>-1.$ Therefore
$$
I_{1}\leq C \left(\int\limits_{a}^{b}\left|(x-a)^{\frac{\beta}{p_{1}}}(b-x)^{\frac{\gamma }{p\,\xi}}I_{a+}^{1/p}f(x)\right|^{p\xi}
 \,dx\right)^{\frac{1}{p\xi}}
=C\left(\int\limits_{a}^{b}\left|I_{a+}^{\alpha}f(x)\right|^{q}(x-a)^{\frac{\beta q}{p_{1}}}(b-x)^{ \gamma  }
 \,dx\right)^{\frac{1}{q}}\leq
 $$
 $$
  \leq
 C\left(\int\limits_{a}^{b}\left|I_{a+}^{\alpha}f(x)\right|^{q}(x-a)^{\frac{\beta q}{p_{1}}}(b-x)^{ \nu  }
 \,dx\right)^{\frac{1}{q}},\,-1<\nu<\gamma.
$$
Applying Theorem 3.10 \cite[p.78]{firstab_lit:samko1987}, we get
$$
\left(\int\limits_{a}^{b}\left|I_{a+}^{\alpha}f(x)\right|^{q}(x-a)^{\frac{\beta q}{p}}(b-x)^{ \nu  }
 \,dx\right)^{\frac{1}{q}}\leq C \|f\|_{L_{p_{1}}(I,\omega)}.
$$
Taking into account that  $\|f\|_{L_{p_{1}}(I,\omega)}\leq C \|f\|_{L_{p}(I\!,\,\omega)},$ we obtain     \eqref{14.1}.

\end{proof}

The results of the monograph \cite{firstab_lit:samko1987} (see Chapter 1) give us a description of the  fractional integral  mapping properties in the space $L_{p}(I,\omega),\,1< p<\infty,\,p\neq1/\alpha,$   where $\omega$ is some power   function.  Actually, the following question is still relevant. What does happen in the case $p=1/\alpha\,?$   In the non-weighted case, the  approach to  this question   is given in the paper \cite{firstab_lit:J.Peetre}. Also, it can be found in a more convenient form in  the monograph   \cite[p.92]{firstab_lit:samko1987}, there   the following  inequality is given
  $$
  \|I^{1/p}_{a+}f\|^{\ast}\leq C\|f\|_{L_{p}(I)},
  $$
  where
 $$
  \|f\|^{\ast}=\sup\limits_{J\subset I}m_{J}f\,,\;m_{J}f=\frac{1}{|J|}\int\limits_{J}|f(x)-f_{J}|dx,\;
  f_{J}=\int\limits_{J}f(x)dx.
 $$
 It is remarkable that   there is  no  mention on the weighted case in the historical review of the  monograph   \cite{firstab_lit:samko1987}.
  In    contrast to  the said  above  approaches,  we  obtain a description of the fractional integral mapping properties  in the space  $L_{p}(I,\omega)$  in terms of the Jacobi series  coefficients. This  approach is principally different  from ones used in \cite{firstab_lit:samko1987}, in  particular  it  allows us  to avoid problems confected with the case $p=1/\alpha.$
  Further, in this section we deal with the   normalized  Jacobi polynomials $p^{ (\alpha,\,\beta) }_{n},\, n\in \mathbb{N}_{0} .$
 \begin{teo}\label{T1}
Suppose
\begin{equation}\label{15}
 \psi\in L_{p}(I,\omega),\,\omega(x)= (x-a)^{\beta}(b-x)^{\gamma},\,\beta,\gamma\in\left[-1/2,1/2\right],\,  M(\beta,\gamma) <p<m(\beta,\gamma);
\end{equation}
then
\begin{equation*}
  I_{a+}^{\alpha} \psi = f ,\,\alpha\in(0,1),
 \end{equation*}
where
$$
f_{m}=\sum\limits_{n=0}^{\infty}(-1)^{n}\psi_{n}A^{\alpha,\beta,\gamma}_{m n},\;m\in \mathbb{N}_{0} .
$$
This theorem can be formulated in the  matrix form
$$
  A^{\alpha,\beta,\gamma}_{+}\times\psi=f\,,\;\;\;\;\sim\;\;\;\;
  \begin{pmatrix}A^{\alpha,\beta,\gamma}_{00}&-A^{\alpha,\beta,\gamma}_{01}&...\\
 A^{\alpha,\beta,\gamma}_{10}&-A^{\alpha,\beta,\gamma}_{11}&...\\
\cdot\\
\cdot\\ \cdot&&...
\end{pmatrix}\times\begin{pmatrix}\psi_{0}\\\psi_{1}\\ \cdot\\ \cdot \\ \cdot \end{pmatrix}= \begin{pmatrix}f_{0}\\f_{1}\\ \cdot\\ \cdot \\ \cdot \end{pmatrix}.
$$
\end{teo}
\begin{proof}
Note, that according to the results of the paper  \cite{firstab_lit:H. Pollard 3} the system of the  normalized Jacobi   polynomials   has    a basis  property in $L_{p}(I,\omega),\,M(\beta,\gamma) <p<m(\beta,\gamma).$   Hence
$$
 \sum\limits_{n=0}^{l}\psi_{n}p_{n}\stackrel{L_{p}(I\!,\,\omega)}{\longrightarrow} \psi\in L_{p}(I,\omega),\;l\rightarrow \infty.
$$
  Using   Lemma \ref{L1}, we obtain
$$
\sum\limits_{n=0}^{l}\psi_{n}I^{\alpha}_{a+}p_{n}  \stackrel{L_{p}(I\!,\,\omega)}{\longrightarrow} I^{\alpha}_{a+}\left(\sum\limits_{n=0}^{\infty}\psi_{n}p_{n}\right)=I^{\alpha}_{a+}\psi,\;l\rightarrow \infty.
$$
Hence
$$
\sum\limits_{n=0}^{l}\psi_{n}\left(I^{\alpha}_{a+}p_{n},p_{m}\right)_{L_{p}(I\!,\,\omega)} \longrightarrow\left(  I^{\alpha}_{a+}\psi,p_{m}\right)_{L_{p}(I\!,\,\omega)},\;l\rightarrow \infty .
$$
Applying   first     formula  \eqref{9}, we obtain
$$
f_{m}=\left( I^{\alpha}_{a+}\psi,p_{m}\right)_{L_{p}(I\!,\,\omega)}
=  \sum\limits_{n=0}^{\infty}(-1)^{n}\psi_{n}A^{\alpha,\beta,\gamma}_{m n}.
$$
  Using denotations    \eqref{11},  we obtain the matrix form   for the statement of this theorem.
 \end{proof}

The following result is  formulated in terms of the Jacobi series coefficients and is devoted to the representation of a function by the fractional integral.
Consider the Abel  equation under  most general assumptions relative to the right part
\begin{equation}\label{16}
 I^{\alpha}_{a+}\varphi =f,\,\alpha\in(0,1).
\end{equation}
  If the next conditions hold
\begin{equation}\label{16.1}
  I^{1-\alpha}_{a+}f\in AC (\bar{I}),\; (I^{1-\alpha}_{a+}f)    (a)=0,
\end{equation}
  then there  exists a unique solution of    equation \eqref{16}   in the class $L_{1}(I)$    (see Theorem 2.1 \cite[p.31]{firstab_lit:samko1987}).
The sufficient conditions for   existence and  uniqueness of the  Abel  equation solution    are established in the  following theorem under the minimum  assumptions relative to the right part of \eqref{16}. In comparison with the ordinary Abel equation, we avoid   imposing  conditions similar to \eqref{16.1}, moreover  we refuse the assumption that the right part is a  Lebesgue integrable function.
\begin{teo}\label{T2}
Suppose  $ \omega(x)=(x-a)^{\beta}(b-x)^{\gamma},\, \beta,\gamma \in[-1/2,1/2] ,\,M(\beta,\gamma)<p< m (\beta,\gamma) ,$ the right part of    equation \eqref{16} such that
\begin{equation}\label{17}
 \left\|D^{\alpha}_{\!a+} S_{k} f \right\|_{L_{p}(I\!,\,\omega)}\! \leq C,\;k\in \mathbb{N}_{0},\;\;\left|\sum\limits_{n=0}^{\infty}(-1)^{n} f_{n}A^{-\alpha,\beta,\gamma}_{mn}\right|\sim m^{-\lambda},\;m\rightarrow \infty,\;\lambda\in[0,\infty);
\end{equation}
then  there exists  a   unique solution of    equation \eqref{16} in     $L_{p}(I,\omega),$  the   solution belongs to   $L_{q}(I,\omega),$  where: $q=p,$  when $0\leq\lambda\leq 1/2 \,;\;q=\max\{p,t\},\,  t<(2s-1)/(s-\lambda),$ when $\,1/2<\lambda<s\;(s=3/2+\max\{\beta,\gamma\});$  $q$ is arbitrary large, when $\lambda\geq\,s.$
 Moreover the solution   is represented by a convergent in $L_{q}(I,\omega)$   series
\begin{equation}\label{18}
\psi(x)=\sum\limits_{m=0}^{\infty}p_{m}(x) \sum\limits_{n=0}^{\infty}(-1)^{n}f_{n}A^{-\alpha,\beta,\gamma}_{mn}.
\end{equation}
\end{teo}
\begin{proof}
  Applying   first formula \eqref{10},   we obtain  the following relation
\begin{equation}\label{19}
\left( D^{\alpha}_{\!a+} S_{k} f ,p_{m}  \right)_{L_{2}(I\!,\,\omega)} \longrightarrow\sum\limits_{n=0}^{\infty}(-1)^{n} f_{n}A^{-\alpha,\beta,\gamma}_{mn},\;k\rightarrow \infty,\,m\in \mathbb{N}_{0},
\end{equation}
  We can  easily  verify that  $M(\beta,\gamma)<p'<m(\beta,\gamma).$ Hence   due to Theorem   A
  \cite{firstab_lit:H. Pollard 3} the system  $\{p_{n}\}_{0}^{\infty}$ has a basis property in the space $L_{p'}(I,\omega).$   Since   relation \eqref{19} holds and  the sequence $\left\{D^{\alpha}_{\!a +}S_{k} f\right\}_{0}^{\infty}$   is bounded in the sense of   norm $L_{p}(I,\omega),$ then due to the  well-known theorem,  we have that the sequence  $\left\{D^{\alpha}_{\!a +} S_{k} f\right\}_{0}^{\infty}$   converges weakly   to some function
 $\psi\in L_{p}(I,\omega).$     Using Theorem 2.4  \cite[p.44]{firstab_lit:samko1987} and  the Dirichlet formula (see Theorem 1.1 \cite[p.9]{firstab_lit:samko1987}),    we get
 $$
\left( S_{k} f ,p_{m} \right)_{L_{2}(I\!,\,\omega)}= \left(I^{\alpha}_{\!a+}  D^{\alpha}_{\!a +} S_{ k} f ,p_{m} \right)_{L_{2}(I\!,\,\omega)} =
\left( D^{\alpha}_{\!a +}S_{k} f ,\omega^{-1}\!I^{\alpha}_{b-}\varphi_{m} \right)_{L_{2}(I\!,\,\omega)}.
 $$
 Let us show that $\omega^{-1}\!I^{\alpha}_{b-}\varphi_{m}\in L_{p'}(I,\omega).$ For this purpose consider the functional
 $$
 l_{1}(f):=\left(   f ,\omega^{-1}I^{\alpha}_{b-}\varphi_{m} \right)_{L_{2}(I\!,\,\omega)}.
 $$
  Using  the H\"{o}lder inequality,   Lemma \ref{1}, we have
 \begin{equation}\label{19.1}
  \left(  I^{\alpha}_{a+} f ,p_{m}  \right)_{L_{2}(I\!,\,\omega)}\leq C\|f\|_{L_{p}(I\!,\,\omega)}\|p_{m}\|_{L_{p'}(I,\omega)}< \infty.
\end{equation}
Hence using the Dirichlet formula, we have
 $$
  \left(  I^{\alpha}_{a+} f ,p_{m}  \right)_{L_{2}(I\!,\,\omega)}= \left(   f ,\omega^{-1}I^{\alpha}_{b-}\varphi_{m} \right)_{L_{2}(I\!,\,\omega)}.
  $$
By virtue of this fact, we can rewrite  relation \eqref{19.1} in the following form $$|l_{1}(f)|\leq C \|f\|_{L_{p}(I\!,\,\omega)},\,\forall f\in L_{p}(I,\omega).$$
Using  the  Riesz representation theorem,
we obtain $\omega^{-1}I^{\alpha}_{b-}\varphi_{m}\in L_{p'}(I,\omega).$
Hence,  we get
 $$
 \left( D^{\alpha}_{\!a +} S_{ k} f ,\omega^{-1}I^{\alpha}_{b-}\varphi_{m} \right)_{L_{2}(I\!,\,\omega)}\rightarrow \left( \psi ,\omega^{-1}I^{\alpha}_{b-}\varphi_{m} \right)_{L_{2}(I\!,\,\omega)}.
 $$
 Using  the  Dirichlet formula, we obtain
 \begin{equation}\label{20}
 \left( \psi ,\omega^{-1}I^{\alpha}_{b-}\varphi_{m} \right)_{L_{2}(I\!,\,\omega)} =
\left(  I^{\alpha}_{a+}\psi ,p_{m}    \right)_{L_{2}(I\!,\,\omega)},\;m\in \mathbb{N}_{0} .
 \end{equation}
Hence
 $$
 \left( S_{k} f ,p_{m} \right)_{L_{2}(I\!,\,\omega)} \longrightarrow \left(  I^{\alpha}_{a+}\psi ,p_{m}    \right)_{L_{2}(I\!,\,\omega)} ,\,k\rightarrow \infty,\;m\in \mathbb{N}_{0} .
 $$
Taking into account that
\begin{equation*}
\left( S_{k} f ,p_{m} \right)_{L_{2}(I\!,\,\omega)}=  \left\{ \begin{aligned}
 f_{m},\;k\geq m,\\
  \!0  ,\;  k<m\, \\
\end{aligned}
 \right.\;\;,
\end{equation*}
we obtain
$$
\left(  I^{\alpha}_{a+}\psi ,p_{m}    \right)_{L_{2}(I\!,\,\omega)}=f_{m},\,m\in \mathbb{N}_{0}.
$$
 Using the  uniqueness property of the  Jacobi  series expansion, we obtain   $I^{\alpha}_{a+}\psi=f$     almost everywhere.
Hence  there  exists a  solution  of the Abel  equation \eqref{16}.   If we assume  that there  exists   another solution $\phi\in L_{p}(I,\omega),$ then we get  $I^{\alpha}_{a+}\psi=I^{\alpha}_{a+}\phi$ almost everywhere. Consider the function  $\eta\in C_{0}^{\infty}(I).$
   Using Theorem 2.4   \cite[p.44]{firstab_lit:samko1987}  and the  Dirichlet formula,  we have
$$
\left(\psi-\phi,\eta \right)_{L_{2}(I)}=\left(\psi-\phi,I^{\alpha}_{b-}D_{b-}^{\alpha}\eta \right)_{L_{2}(I)}=\left(I^{\alpha}_{a+}[\psi-\phi],D_{b-}^{\alpha}\eta \right)_{L_{2}(I)}=0.
$$
  Consider an interval $I'\subset I.$  Note that  $\psi,\phi\in L_{p }(I'),\, \forall I'.$ Since   $C_{0}^{\infty}(I')\subset C_{0}^{\infty}(I ),$ here we assume that  the functions belonging  to $C_{0}^{\infty}(I')$ have the zero    extension outside of $I',$  then we obtain
 $$
\left(\psi-\phi,\eta \right)_{L_{2 }(I')}=0,\,\forall \eta \in C_{0}^{\infty}(I').
$$
We claim that  $\psi\neq\phi.$ Hence  in accordance  with the consequence of the Hahn-Banach theorem there exists the element $\vartheta\in L_{p'}(I'),$ such that
$$
 \left(\psi-\phi,\vartheta \right)_{L_{2 }(I')} =\|\psi-\phi\|_{L_{p }(I')}>0.
$$
On the other hand, there exists  the sequence $\{\eta_{n}\}_{1}^{\infty}\subset C_{0}^{\infty}(I'),$ such  that $\eta_{n}\stackrel{L_{p'}(I' )}{\longrightarrow}\vartheta.$ Hence
$$
0=\left(\psi-\phi,\eta_{n} \right)_{L_{2 }(I')}\rightarrow \left(\psi-\phi,\vartheta \right)_{L_{2 }(I')}.
$$
Thus we come to contradiction. Hence    $\psi=\phi$ almost everywhere  on $I',\,\forall I'\subset I. $ It implies
that $\psi=\phi$ almost everywhere  on $I.$  The    uniqueness has been proved. Now let us proceed to the following part of the  proof.
Note  that  it was proved above  $\psi \in L_{p}(I,\omega), $   when $0\leq\lambda<\infty.$
Let us show that   $\psi \in L_{q}(I,\omega), $ where $  q< (2s-1)/ (s-\lambda) ,\,1/2<\lambda<s.$   In accordance with the  reasoning given above,      we have
\begin{equation*}
\left(D^{\alpha}_{\!a +} S_{k} f ,p_{m} \right)_{L_{2}(I\!,\,\omega)} \longrightarrow \left( \psi,p_{m} \right)_{L_{2}(I\!,\,\omega)},\;m\in \mathbb{N}_{0}.
\end{equation*}
Combining this fact with   \eqref{19}, we get
\begin{equation}\label{21}
\psi_{m}=\left( \psi,p_{m} \right)_{L_{2}(I\!,\,\omega)}  =\sum\limits_{n=0}^{\infty}(-1)^{n}f_{n}A^{-\alpha,\beta,\gamma}_{mn},\;m\in \mathbb{N}_{0}.
\end{equation}
Using the  theorem conditions,  we have
$$
|\psi_{m}|\sim m^{-\lambda},\,m\rightarrow \infty.
$$
Now we need an adopted version of the    Zigmund-Marczinkevich theorem  (see \cite{firstab_lit: Marz}), which  establishes the following.
Let $\{\phi_{n}\}$ be an orthoghonal system on the segment  $\bar{I}$ and $\|\phi_{n}\|_{L_{\infty}(I)}\leq M_{n},\,(n=1,2,...),$ where $M_{n}$ is a monotone increasing sequence of   real numbers. If $q\geq2$ and we have
\begin{equation}\label{21a}
\Omega_{q}(c)=\left(\sum\limits_{n=1}^{\infty}|c_{n}|^{q}n^{q-2}M_{n}^{q-2}\right)^{1/q}<\infty,
\end{equation}
then the series
$
\sum_{n=1}^{\infty}c_{n}\phi_{n}(x)
$
converges in  $ L_{q}(I)$ to some function  $f\in L_{q}(I)$  and $\|f\|_{L_{q}(I)}\leq C\Omega_{q}(c).$
We aim to apply this theorem   to the case of the Jacobi system, however  we need some auxiliary reasoning. As the matter of fact,   we deal with the weighted $L_{p}(I,\omega)$ spaces, but the Zigmund-Marczinkevich theorem in its pure form formulated in terms of  the non-weighted case. Consider the following change of the variable $\int _{a}^{x}\omega(t)dt=\tau.$ For the  solution $\psi\in L_{p}(I,\omega) ,$ we have
\begin{equation}\label{21b}
\psi_{n}=\int\limits_{a}^{b}\psi(x)p_{n}(x)\omega(x)dx=\int\limits_{0}^{B}\tilde{\psi}(\tau)\phi_{n}(\tau) d\tau,\,
\end{equation}
where $\tilde{\psi}(\tau)=\psi(\kappa(\tau)),\,\phi_{n}(\tau)=p_{n}(\kappa(\tau)),\,\kappa(\tau)=(b-a)^{-(\beta+\gamma+1)}B^{-1}_{\tau}(\beta+1,\gamma+1),\, B=(b-a)^{\beta+\gamma+1}B(\beta+1,\gamma+1).$
Hence, if we note the    estimate    $|p_{n}(x)|\leq C n^{a+1/2},\,a=\max\{\beta,\gamma\},\,x\in \bar{I} $ (see Theorem 7.3 \cite[p.288]{firstab_lit:Suetin}), then due to the  change of  the variable, we have $ |\phi_{n}(\tau)|\leq  V_{n},  \,\tau\in[0,B],\,V_{n}=C n^{a+1/2}.$ Also, it is clear that $(\phi_{m},\phi_{n})_{L_{2}(0,B)}=\delta_{mn},$ where $\delta_{mn}$ is the    Kronecker symbol. Thus $\{\phi_{m}\}_{0}^{\infty}$ is the   orthonormal system on $[0,B]$ that  satisfies   the  conditions of the Zigmund-Marczinkevich theorem.
 It can easily be checked  that due to the    theorem  conditions  the following series is convergent
\begin{equation}\label{22}
     \sum_{m=0}^{\infty} m^{q\,(s -\lambda)-2s}  <\infty,\,1/2<\lambda<s,\,q<(2s-1)/(s-\lambda).
\end{equation}
For the values   $\lambda\geq s,$  series \eqref{22} converges for an  arbitrary positive $q.$
In accordance with given above, we have
$$
 \left\{\sum_{m=0}^{\infty}|\psi_{m}|^{q}m^{q-2}V^{q-2}_{m}\right\}^{1/q}\leq C  \left\{\sum_{m=0}^{\infty} m^{q\,(s -\lambda)-2s}\right\}^{1/q}<\infty.
$$
Thus all conditions of the  Zigmund-Marczinkevich theorem   are fulfilled. Hence, we can conclude  that there exists a function $\nu$ such  that the next estimate holds
\begin{equation}\label{23}
  \|\nu\|_{L_{q}(0,B)} \leq C  \left\{\sum_{m=0}^{\infty}|\nu_{m}|^{q}m^{q-2}M^{q-2}_{m}\right\}^{1/q}<\infty.
\end{equation}
  Since the system $\{p_{m}\}_{0}^{\infty}$ has a basis property in   $L_{p}(I,\omega),$ then  it is not  hard to prove  that the system $\{ \phi_{m} \}_{0}^{\infty}$ has a basis property in   $L_{p}(0,B).$ Since the functions $\nu$  and $\tilde{\psi}$ have the same Jacobi series  coefficients, then  $\nu=\tilde{\psi}$ almost everywhere  on $(0,B).$    By virtue of the chosen change of the variable, we obtain  $\|\psi\|_{L_{q}(I,\omega)}=\|\tilde{\psi}\|_{L_{q}(0,B) }.$ Consequently,  the solution    $\psi$  belongs to the space $L_{q}(I,\omega),\, q<(2s-1)/(s-\lambda),$ when $\,1/2<\lambda<s$ and   the index   $q$ is arbitrary large, when $\lambda\geq s.$ Taking into account \eqref{21b}   and applying the      Zigmund-Marczinkevich theorem, we have
$$
 \sum\limits_{m=0}^{k}\phi_{m} \psi_{m} \stackrel{L_{q}(0,B)}{\longrightarrow} \tilde{\psi},\,k \rightarrow \infty.
$$
Using the inverse  change of the variable and applying   \eqref{21}, we obtain
   \eqref{18}.
 \end{proof}

\subsection{Non-simple property problem}

The questions related  to existence of such an invariant subspace of the operator that the operator restriction to the subspace  is selfadjont   (the so-called non-simple property \cite[p.275]{firstab_lit:1Gohberg1965} ) are still relevant for today. Thanks to the powerful tool provided by   the   Jacobi polynomials theory, we are able to approach a little close to solving   this  problem for the Riemann-Liouville operator.

In this section we deal with the  so-called normalized  ultraspherical polynomials $p^{  (\beta,\,\beta) }_{\,n}(x)$
in the   weighted space $L_{p}(I,\omega),\, \omega(x)=\left[(x-a) (b-x)\right]^{\beta}\!\!,\,\beta\geq -1/2 ,\,1\leq p<\infty .$ In accordance with  \cite{firstab_lit:H. Pollard 2}     the system of the  normalized ultraspherical polynomials    has   a basis property in    $L_{p}(I,\omega),$ if
$
 1-1/(3+2\beta)<p/2<1+ 1/(1+2\beta),\,\lambda=\beta+1/2
$
and does not have a basis property, if $1/2\leq p/2<1-1/(3+2\beta) $ or $p/2>1+ 1/(1+2\beta)  .$
Having noticed that $A^{\alpha,\beta,\beta}_{mn}=A^{\alpha,\beta,\beta}_{nm},\,m,n\in \mathbb{N}_{0},$ using  formulas   \eqref{9}, we obtain
\begin{equation}\label{24}
\int\limits_{a}^{b}  \left(I^{\alpha}_{a+} p _{n}\right)(x) p _{m}(x) \omega(x)dx=(-1)^{n+m}\int\limits_{a}^{b} p _{n}(x)\left(I^{\alpha}_{a+} p _{m}\right)(x)\omega(x)dx\,;
$$
$$
\int\limits_{a}^{b}  \left(I^{\alpha}_{b-} p _{n}\right)(x) p _{m}(x) \omega(x)dx=(-1)^{n+m}\int\limits_{a}^{b} p _{n}(x)\left(I^{\alpha}_{b-} p _{m}\right)(x)\omega(x)dx,\;\;m,n\in \mathbb{N}_{0}.
\end{equation}
Taking into account these formulas we   conclude that the fractional integral operator  is symmetric
  in the subspaces of $L_{2}(I,\omega)$ generated respectively by   the  even system $\{p_{2k }\}_{0}^{\infty}$  and  the odd system $\{p_{2k+1}\}_{0}^{\infty}$  of the  normalized  ultraspherical  polynomials.
 Let us denote    by $L^{+}_{2}(I,\omega),\,L^{-}_{2}(I,\omega)$ these subspaces respectively.
The following theorem gives  us an alternative.
\begin{teo}\label{T3} $(\mathrm{Alternative})$ Suppose $\alpha\in (1/2,3/2),\,\omega(x)=(x-a)^{\beta}(b-x)^{\beta},\,\alpha-1/2<\beta <1;$ then we have the following alternative:
  Either the  fractional integral  operator     acting in $L_{2}(I,\omega)$   is non-simple or one  has  an infinite  sequence   of  the   included  invariant subspaces having   the non-zero  intersection with  both subspaces $L^{+}_{2}(I,\omega),\,L^{-}_{2}(I,\omega).$
 \end{teo}
\begin{proof}
We   provide the proof only for the left-side case,   since the proof corresponding to the right case is   analogous  and can be obtained by     simple   repetition.
Let us show that the operator $I^{\alpha}_{a+}:L_{2}(I,\omega)\rightarrow L_{2}(I,\omega)$ is compact.
   Using  Theorem 3.12 \cite[p.81]{firstab_lit:samko1987}, we have the estimate
\begin{equation}\label{25}
\|I_{a+}^{\alpha}f\|_{H_{0}^{\lambda}(\bar{I},r)}\leq C\|f\|_{L_{2}(I\!,\,\omega)},\,\lambda=\alpha-1/2,
\end{equation}
where $r(x)=(x-a)^{\beta/2}(b-x)^{\beta/2},$ if $\beta>2\alpha-1$ and  $r(x)=(x-a)^{\beta/2}(b-x)^{\alpha-1/2+\delta}$ for sufficiently  small $\delta>0,$ if $\beta\leq2\alpha-1.$
 It can easily be checked that   in the case $(\beta>2\alpha-1),$ we have
$$
 \left(\int\limits_{a}^{b}\left|(I_{a+}^{\alpha})f (x)  \right|^{2} \omega(x)dx\right)^{1/2}=\left(\int\limits_{a}^{b}\left|r(x)(I_{a+}^{\alpha})f (x)  \right|^{2}  dx\right)^{1/2}
\leq C \|I_{a+}^{\alpha}f\|_{H^{\lambda}_{0}(\bar{I},r)},
$$
and in the case $(\beta\leq2\alpha-1),$ we have
$$
 \left(\int\limits_{a}^{b}\left|(I_{a+}^{\alpha})f (x)  \right|^{2} \omega(x)dx\right)^{1/2}=\left(\int\limits_{a}^{b}\left|r(x)(I_{a+}^{\alpha})f (x)   \right|^{2}  (b-x)^{\beta-(2\alpha-1+2\delta)}dx\right)^{1/2}
.$$
Note that $\beta-(2\alpha-1+2\delta)>-1$ for sufficiently small $\delta.$
Therefore
$$
\left(\int\limits_{a}^{b}\left|(I_{a+}^{\alpha})f (x) r(x)  \right|^{2}  (b-x)^{\beta-(2\alpha-1+2\delta)}dx\right)^{1/2}\leq C \|I_{a+}^{\alpha}f\|_{H^{\lambda}_{0}(\bar{I},r)}.
$$
Thus,   using estimate \eqref{25}, we obtain
\begin{equation}\label{26}
\|I_{a+}^{\alpha}f\|_{L_{2}(I\!,\,\omega)}\leq C\|f\|_{L_{2}(I\!,\,\omega)}.
\end{equation}
Now let us use   the Kolmogorov criterion of compactness (see \cite{firstab_lit: Kolmogoroff  A. N.}),   which claims  that a set in the space $L_{p}(I,\omega),\,1\leq p<\infty$ is compact, if this set  is bounded and equicontinuous with respect to the norm $L_{p}(I,\omega).$  Note that    by virtue of  \eqref{26} the set $I_{a+}^{\alpha}(\mathfrak{N})$ is bounded in $L_{2}(I,\omega),$ where  $\mathfrak{N}:=\left\{f: \|f\|_{L_{2}(I\!,\,\omega)}\leq M,\, M>0\right\}.$    Using  \eqref{25},  we get  $\| I_{a+}^{\alpha}f\|_{H_{0}^{\lambda}(\bar{I}\!,\,r)}\leq C_{1},\,\forall f\in \mathfrak{N}.$ Hence  in accordance with the definition,   we have
 $|(I_{a+}^{\alpha}f)(x+t)r(x+t)-(I_{a+}^{\alpha}f)(x)r(x) | < C_{1} t^{\lambda},\,\forall f\in \mathfrak{N},\, \forall x\in [a,b), $ where     $t$ is a sufficiently small positive number.   Under the assumption that     functions have a zero extension  outside  of $\bar{I},$ we have
$$
\left\{\int\limits_{a}^{b}\left|(I_{a+}^{\alpha}f)(x+t) -(I_{a+}^{\alpha}f)(x) \right|^{2}\omega(x)dx\right\}^{\!\!\frac{1}{2}}\leq\left\{\int\limits_{b-t}^{b}\left| (I_{a+}^{\alpha}f)(x) \right|^{2}\omega(x)dx\right\}^{\!\!\frac{1}{2}}+
$$
$$
+\left\{\int\limits_{a}^{b-t}\left|(I_{a+}^{\alpha}f)(x+t) -(I_{a+}^{\alpha}f)(x) \right|^{2}\omega(x)dx\right\}^{\!\!\frac{1}{2}}=\tilde{I} +I.
$$
Assume that $f\in \mathfrak{N}$ and consider  the case $(\beta\leq 2\alpha-1).$
Note that due to  Theorem 3.12 \cite[p.81]{firstab_lit:samko1987},  we obtain
$$
\tilde{I}=\left\{\int\limits_{b-t}^{b}\left| (I_{a+}^{\alpha}f)(x) r(x)\right|^{2}(b-x)^{\beta-\mu}dx\right\}^{\!\!\frac{1}{2}}\leq C_{1}\left\{\,\int\limits_{b-t}^{b} (b-x)^{\beta-\mu}dx\right\}^{\!\!\frac{1}{2}}\leq C t^{\beta-\mu+1},
$$
where $ \mu=2\alpha-1+2\delta,\,r(x)=(x-a)^{\beta/2}(b-x)^{\mu/2}.$
Using the Minkowski inequality, we get
$$
I\leq\left\{\int\limits_{a}^{b-t}\left|(I_{a+}^{\alpha}f)(x+t)r(x+t) -(I_{a+}^{\alpha}f)(x)r(x) \right|^{2}(b-x)^{\beta-\mu}dx\right\}^{\!\!\frac{1}{2}}\!\!\!+
$$
$$
+\left\{\int\limits_{a}^{b-t}\left|(I_{a+}^{\alpha}f)(x+t)[r(x+t) - r(x)] \right|^{2}(b-x)^{\beta-\mu}dx\right\}^{\frac{1}{2}}=I_{1}+I_{2},\,
$$
As before, applying    Theorem 3.12 \cite[p.81]{firstab_lit:samko1987}, we get
$
I_{1}\leq C t^{\alpha-1/2}.
$
  Using the inequality $(\tau+1)^{\nu}<\tau^{\nu}+1,\,\tau>1,\,0<\nu<1,$  we obtain
$$
\left|(x+t-a)^{\beta/2}-(x-a)^{\beta/2}\right|= t^{\beta/2}\left|\left(\frac{x-a}{t}+1\right)^{\beta/2}-\left(\frac{x-a}{t}\right)^{\beta/2}\right|< t^{\beta/2},\,a+t<x<b.
$$
In the same way, using the inequality $(\tau-1)^{\nu}>\tau^{\nu}-1,\,\tau>1,\,0<\nu<1,$ we get
$$
\left|(b-x-t)^{\mu/2}-(b-x)^{\mu/2}\right|< t^{\mu/2} ,\,a<x<b-t.
$$
Since   $r(x)$ is a product of the functions that satisfy the  H\"{o}lder condition,   then it is not hard to prove that
$$
|r(x+t)-r(x)|< C_{2}t^{\beta/2},\,a+t<x<b-t.
$$
Using the fact  $ \| I_{a+}^{\alpha}f\|_{H_{0}^{\lambda}(\bar{I},r }\leq C_{1},$ we get
$$
 I^{2}_{2}\leq C_{1}  \int\limits_{a}^{b-t}r^{-2}(x+t)\left|  r(x+t) - r(x)  \right|^{2}(b-x)^{\beta-\mu}dx  \leq
 $$
 $$
\leq C_{1}  \int\limits_{a+t}^{b-t}r^{-2}(x+t)\left|  r(x+t) - r(x)  \right|^{2}(b-x)^{\beta-\mu}dx+
 $$
 $$
 + C_{1}\int\limits_{a}^{a+t}r^{-2}(x+t)\left|  r(x+t) - r(x)  \right|^{2}(b-x)^{\beta-\mu}dx =
 I_{21}+I_{22}.
$$
Taking into account the above reasoning, we have
$$
  I_{21}\leq C t^{\beta } \int\limits_{a+t}^{b-t}r^{-2}(x+t) (b-x)^{\beta-\mu}dx \leq C t^{2\beta-\mu }  \int\limits_{a+t}^{b-t}(x-a+t)^{-\beta }(b-t-x)^{ -\mu  }       dx =
$$
$$
= C t^{2\beta-\mu } \left\{(b-a)^{1-\beta-\mu} B(1-\beta,1-\mu)-\int\limits_{a-t}^{a+t}(x-a+t)^{-\beta  }(b-t-x)^{ -\mu  }       dx\right\}^{\frac{1}{2}}\leq C t^{2\beta-\mu}\,;
$$
$$
 I_{22}\leq C  \int\limits_{a}^{a+t}(x-a+t)^{-\beta  }(b-t-x)^{ -\mu  }  (b-x)^{\beta-\mu}     dx \leq C  \int\limits_{a}^{a+t}(x-a+t)^{-\beta  }      dx \leq C t^{ 1-\beta  }.
$$
Hence we conclude that $I_{2}\leq C t^{\delta_{1}}$  and as a consequence, we obtain $I\leq C t^{\delta_{2}},$ where $\delta_{1},\delta_{2}$ are some positive numbers.
 To achieve   the case $(\beta>2\alpha-1)$ we should just repeat  the previous reasoning having replaced     $\mu$ by  $\beta.$  The proof is omitted. Thus  in both cases considered above  we obtain
 $$
\forall\varepsilon>0,\,\exists\, t:=t(\varepsilon) :\, \| (I_{a+}^{\alpha}f)(\cdot+t)  -(I_{a+}^{\alpha}f)(\cdot)\|_{L_{2}(I\!,\,\omega)}<\varepsilon,\,\forall f\in \mathfrak{N}.
 $$
It implies that  the conditions of the Kolmogorov criterion of compactness \cite{firstab_lit: Kolmogoroff  A. N.} are fulfilled. Hence  any  bounded set  with respect to the  norm $L_{2}(I,\omega)$ has a compact image. Therefore the operator $I_{a+}^{\alpha}:L_{2}(I,\omega)\rightarrow L_{2}(I,\omega)$ is compact.
Now  applying the von Neumann  theorem \cite[p.204]{firstab_lit: Ahiezer1966}, we   conclude that there exists a non-trivial invariant subspace of the operator $I_{a+}^{\alpha},$  which we   denote by $\mathfrak{M}.$ On the other hand, using the    basis property of the  system $\{p_{n}\}_{0}^{\infty}$, we have
 $L_{2}(I,\omega) =L^{+}_{2}(I,\omega) \oplus L^{-}_{2}(I,\omega).$    It is quite sensible to assume  that
$ \mathfrak{M}\cap L^{+}_{2}(I,\omega)\neq 0,\; \mathfrak{M}\cap L^{-}_{2}(I,\omega)\neq 0.$    If we assume  otherwise, then   we have   an invariant subspace   on which the operator  $I_{a+}^{\alpha}$,
 by virtue of   formulas \eqref{24},  is  selfadjoin   and    we get  the first statement of the alternative. Continuing this line of reasoning, we see that  under  the assumption excluding the first statement of the alternative we come to conclusion that this process  can be finished only in the case, when on   some  step we get a finite-dimensional   invariant subspace. We claim  that it can not be! The proof is by {\it reductio ad absurdum}. Assume the converse, then  we obtain a finite-dimensional restriction  $ \tilde{I} _{a+}^{\alpha}$ of the operator $I_{a+}^{\alpha}.$
Applying the reasoning of   Theorem 2, we can easily  prove that  the point zero is not an eigenvalue of the operator
 $ I_{a+}^{\alpha},$ hence one is not an eigenvalue of the operator
 $ \tilde{I}_{a+}^{\alpha}.$
    It implies that   in accordance with the fundamental theorem of algebra  the operator $\tilde{I} _{a+}^{\alpha}$ has at least one non-zero    eigenvalue (since  $\tilde{I} _{a+}^{\alpha}$ is finite-dimensional). It is clear  that this eigenvalue is an eigenvalue of the operator $I_{a+}^{\alpha}.$   We can write
  \begin{equation}\label{27.0}
  \exists \lambda\in \mathbb{C},\,\lambda\neq 0,\,f\in L_{2}(I,\omega),\,f\neq 0:\;I_{a+}^{\alpha}f=\lambda f\;\;\mathrm{a.e.}
 \end{equation}
 Further,  we use the method described in \cite[p.14]{firstab_lit:Tricomi}.
 Using the Cauchy Schwarz inequality, we get
 \begin{equation}\label{27}
 |f(x)|^{2}\leq |\lambda|^{-2} B(x)   \int\limits_{a}^{x}|f(t)|^{2}\omega(t)dt\leq |\lambda|^{-2}\|f \|^{2}_{L_{2}(I\!,\,\omega)}B(x),
 \end{equation}
 where
 $$
 B(x)=\Gamma^{-1}(\alpha) \int \limits_{a}^{x}(x-t)^{2\alpha-2} \omega^{-1}(t)dt.
 $$
Substituting $f(t)$ for  $|\lambda|^{-2}\|f \|^{2}_{L_{2}(I\!,\,\omega)}B(x)$ in \eqref{27},  we get
$$
|f(x)|^{2}\leq \|f \|^{2}_{L_{2}(I\!,\,\omega)} |\lambda|^{-4}B(x) \int\limits_{a}^{x}B(t)\omega(t)dt.
$$
Continuing  this process, we obtain
$$
|f(x)|^{2}\leq \|f \|^{2}_{L_{2}(I\!,\,\omega)} |\lambda|^{-2(n+1)}B(x) \underbrace{\int\limits_{a}^{x}B(x_{n})\omega(x_{n})dx_{n}\int\limits_{a}^{x_{n}}B(x_{n-1})\omega(x_{n-1})dx_{n-1}...\int\limits_{a}^{x_{2}}}_{\text{n  integrals} }B(x_{1 })\omega(x_{1 }) dx_{1},\,n\in \mathbb{N}.
$$
Let
$$
B_{n}(x):= \int\limits_{a}^{x}B(x_{n})\omega(x_{n})dx_{n}\int\limits_{a}^{x_{n}}B(x_{n-1})\omega(x_{n-1})dx_{n-1}...\int\limits_{a}^{x_{2}} B(x_{1 })\omega(x_{1 }) dx_{1},
$$
thus
\begin{equation}\label{28}
B_{n}(x)= \int\limits_{a}^{x}B(t)B_{n-1}(t)\omega(t)dt,\,B_{0}(x):=1,\,n\in \mathbb{N}.
\end{equation}
Let us show that
$
B_{n}(x)= B^{n}_{1}(x)/n!.
$
It is obviously true  in the case $(n=1).$   Assume that the relation  $
B_{n-1}(x)= B^{n-1}_{1}(x)/(n-1)!
$   is fulfilled   and let us  deduce that
$
B_{n}(x)= B^{n}_{1}(x)/n!.
$
Using   \eqref{28}, we obtain
$$
B_{n}(x)= \frac{1}{(n-1)!}\int\limits_{a}^{x}B(t)B^{n-1}_{1}(t)\omega(t)dt=\frac{1}{(n-1)!}\int\limits_{a}^{x} B^{n-1}_{1}(t)\frac{d B_{1}(t)}{dt} dt= \frac{B^{n }_{1}(x)}{n!}.
$$
Hence
$$
|f(x)|^{2}\leq  \frac{1}{n!}\|f \|^{2}_{L_{2}(I\!,\,\omega)}|\lambda|^{-2(n+1)} B(x)   B^{n }_{1}(x)  ,\,n\in \mathbb{N}.
$$
Using the  Dirichlet formula, we get
$$
  B _{1}(x) \leq  \frac{1}{\Gamma(\alpha)}\int \limits_{a}^{b}\omega(y)dy \int \limits_{a}^{y}(y-t)^{2\alpha-2} \omega^{-1}(t)dt =
  \frac{1}{\Gamma(\alpha)}\int \limits_{a}^{b}\omega^{-1}(t)dt \int \limits_{t}^{b}(y-t)^{2\alpha-2}\omega(y)dy=:J.
$$
By virtue of the  theorem  conditions,    we   conclude that $J<\infty.$ Hence, we have
$$
|f(x)|^{2}\leq  \frac{|\lambda|^{-2(n+1)}J^{n}}{n!}   \|f \|^{2}_{L_{2}(I\!,\,\omega)} B(x) ,\,n\in \mathbb{N}.
$$
Since   $ \left(|\lambda|^{-2(n+1)}J^{n}\right)/ n! \rightarrow 0,\;n\rightarrow \infty,$ then  $f(x)=0,\;x\in I.$
We have obtained the  contradiction with     \eqref{27.0}, which allows us to conclude that   there  does   not exist  a finite dimensional invariant subspace.
 It implies that  we have the sequence of the included invariant   subspaces
  $$
 \mathfrak{M}_{1}\supset\mathfrak{M}_{2}\supset...\supset\mathfrak{M}_{k}\supset...\;\,,
 $$
 $$
 \mathfrak{M}_{k}\cap L^{+}_{2}(I,\omega)\neq 0,\; \mathfrak{M}_{k}\cap L^{-}_{2}(I,\omega)\neq 0,\;k=1,2,... \;\,.
 $$
  \end{proof}
\section{Conclusions}
In this paper, the first our aim  is  to reformulate  in terms of the Jacoby series  coefficients  the previously known theorems describing the Riemann-Liouville operator action in the weighted   spaces    of Lebesgue  p-th power  integrable functions, the second aim is  to approach a little bit  closer to  solving   the problem: whether the Riemann-Liouville operator acting  in the weighted  space of Lebesgue square integrable functions is simple.  The approach, which was used in the paper  is in the following: to use the    Jacobi polynomials special properties   that    alow us  to apply   novel methods of  functional analysis   and
   theory of functions of a real variable, which are rather    different   in comparison with the  perviously applied methods  for studying
    the Riemann-Liouville operator. Besides the main results of the paper, we stress that  there was arranged some systematization of the previously known facts of the Riemann-Liouville operator action in the weighted  spaces of Lebesgue  p-th power  integrable functions, when the weighted function is represented by some kind of  a power function. It should be noted that the previously known description of  the Riemann-Liouville operator action in the weighted  spaces of Lebesgue  p-th power  integrable functions consists of some theorems in which the conditions imposed on the weight   function  have the gaps,  i.e. some cases corresponding to the concrete   weighted functions was not considered.    Motivated by this, among the unification of the known results,  we managed to fill the gaps of the conditions and formulated this result as  a separate  lemma.
   The following main  results were obtained in terms of the Jacoby series coefficients: the theorem on the Riemann-Liouville operator direct action was proved, the existence and uniqueness theorem for Abel equation in the weighted  spaces of Lebesgue  p-th power  integrable functions  was proved and the solution formula was given,
    the alternative  in accordance with which the Riemann-Liouville operator is either simple or one  has the sequence of the included invariant subspaces   was established.     Note that these  results give us such  a view of the fractional calculus that has a lot of     advantages.   For instance,   we can reformulate Theorem \ref{T2} under more general assumptions relative to the integral operator on the left side of equation \eqref{16}, at the same time having   preserved the main scheme of the reasonings. In this case the most important problem may be, in what way we are able to  calculate the coefficients given by formula  \eqref{10}. Besides,   the notorious case $p=1/\alpha,$  which was successfully achieved in this paper  is also worth noticing.  Thus the obtained results make a prerequisite of researching in such   a   direction of fractional calculus.
\subsection*{Acknowledgments}
  Gratitude is expressed to Boris G. Vaculov    for   valuable remarks and comments.

\end{document}